\documentclass{aptpub}

\usepackage[utf8x]{inputenc}
\usepackage[a4paper]{geometry}

\authornames{J. SEGERS} % insert the authors here for use in running head
\shorttitle{One- versus multi-component regular variation and extremes of Markov trees} % insert short title here for use in running head

\usepackage{dsfont}
\usepackage{enumerate}
\usepackage{tikz}
\usepackage{paralist}
\RequirePackage[numbers]{natbib}

\usepackage{color}
\definecolor{chocolate(traditional)}{rgb}{0.48, 0.25, 0.0}
\definecolor{burntorange}{rgb}{0.8, 0.33, 0.0}
\definecolor{ceruleanblue}{rgb}{0.16, 0.32, 0.75}
\newcommand{\js}[1]{\textcolor{chocolate(traditional)}{\sffamily\small [JS: {#1}]}}
\newcommand{\tree}{\mathcal{T}}
\newcommand{\law}{\mathcal{L}}

\renewcommand{\SS}{\mathbb{S}}

\newcommand{\Mz}{\mathcal{M}_0}

\newcommand{\Mzi}{\mathcal{M}_{0,i}}
\newcommand{\MzI}{\mathcal{M}_{0,I}}
\newcommand{\Czi}{\mathcal{C}_{0,i}}
\newcommand{\CzI}{\mathcal{C}_{0,I}}
\newcommand{\Szi}{\SS_{0,i}}
\newcommand{\SzI}{\SS_{0,I}}

\newcommand{\dto}{\raisebox{-0.5pt}{\,\scriptsize$\stackrel{\raisebox{-0.5pt}{\mbox{\scriptsize $d$}}}{\longrightarrow}$}\,}
\newcommand{\zto}{\raisebox{-0.5pt}{\,\scriptsize$\stackrel{\raisebox{-0.5pt}{\mbox{\scriptsize $0$}}}{\longrightarrow}$}\,}

\newcommand{\expec}{\operatorname{\mathbb{E}}}
\renewcommand{\Pr}{\operatorname{\mathbb{P}}}
\newcommand{\Fbar}{\overline{F}}
\newcommand{\1}{\mathds{1}}
\newcommand{\point}{\,\cdot\,}
\newcommand{\indep}{\perp\!\!\!\perp}

\newcommand{\abs}[1]{\lvert{#1}\rvert}
\newcommand{\reals}{\mathbb{R}}
\newcommand{\eps}{\varepsilon}
\newcommand{\diff}{\mathrm{d}}

\newcommand{\an}{\operatorname{an}}
\newcommand{\pa}{\operatorname{pa}}
\newcommand{\path}[2]{[{#1}\rightsquigarrow{#2}]}

\newcommand{\Pa}{\operatorname{Pa}}

\renewcommand{\ge}{\geqslant}

\renewcommand{\le}{\leqslant}

\begin{document}

\title{One- versus multi-component regular variation\\and extremes of Markov trees} % insert title - use \\ if it requires more than one line.

\authorone[UCLouvain]{Johan Segers}
\addressone{UCLouvain, LIDAM/ISBA, Voie du Roman Pays 20, B-1348 Louvain-la-Neuve, Belgium. Email: johan.segers@uclouvain.be}

\begin{abstract}
	A Markov tree is a random vector indexed by the nodes of a tree whose distribution is determined by the distributions of pairs of neighbouring variables and a list of conditional independence relations.
	Upon an assumption on the tails of the Markov kernels associated to these pairs, the conditional distribution of the self-normalized random vector when the variable at the root of the tree tends to infinity converges weakly to a random vector of coupled random walks called tail tree. 
%The increment distributions are induced by the tail asymptotics of the Markov kernels along the edges of the tree whereas 
%The coupling is induced by the common stretches that paths starting from the same root node but with different destination nodes have in common. 
	If, in addition, the conditioning variable has a regularly varying tail, the Markov tree satisfies a form of one-component regular variation.
	Changing the location of the root, that is, changing the conditioning variable, yields a different tail tree.
	When the tails of the marginal distributions of the conditioning variables are balanced, these tail trees are connected by a formula that generalizes the time change formula for regularly varying stationary time series.
	The formula is most easily understood when the various one-component regular variation statements are tied up to a single multi-component statement.
	The theory of multi-component regular variation is worked out for general random vectors, not necessarily Markov trees, with an eye towards other models, graphical or otherwise.
%We study the limiting distribution of a random vector conditionally on the event that a given variable exceeds a high threshold. Although the limit depends on the variable concerned by the conditioning event, the various limits are connected through an identity that resembles the time-change formula for regularly varying stationary time series. Moreover, existence of such limits is a sufficient \js{and necessary?} condition for the random vector to be multivariate regularly varying. When specialized to Markov trees, the limiting distribution takes the form of a collection of coupled multiplicative random walks. The distributions of the multiplicative increments are determined by the tails of the pair distributions of connected nodes. Changing the root amounts to changing the directions of certain edges, whose increment distributions are then to be transformed in a specific way. 
\end{abstract}

 % insert keywords separated by a semicolon

%\ams{60G70}{} % insert the primary Maths Subject Classification number in the first bracket
         % and the secondary ams number(s) in the second bracket
         % e.g. \ams{60E20}{49G03;49F10}

% =========================
\section{Introduction} % Initial capital letter, then lower case. No full stop.

Imagine a random vector $X = (X_1, \ldots, X_d)$ of nonnegative variables. One of the components, say $X_i$, is known to have exceeded a large threshold. How does this information affect the conditional distribution of the whole vector $X$? There could be a causal link from $X_i$ to the other variables $X_j$, perhaps via a network of dependence relations, so that tampering with $X_i$ would affect the whole system. Another possibility is that a large value of $X_i$ is merely the result of a large value of some other variable $X_j$. The latter event, however, could have consequences for still other variables $X_k$.

Depending on which one of the $d$ components is known to have been exceptionally large, the conditional distribution of $X$ is likely to be different. Still, if high values of two variables $X_i$ and $X_j$ are not unlikely to arrive together, the conditional distribution of $X$ given that $X_i$ is large must be connected to the one given that $X_j$ is large.

In this paper, these questions are studied for general random vectors using the language of regular variation. The answers are worked out for the particular case that $X$ is a Markov tree. A large value at a particular node is found to spread through the tree via independent increments along the edges. The joint limit distribution is the one of a vector of coupled geometric random walks. The couplings occur through the common edges of different paths starting at the same root node.

Graphical models, of which Markov trees are a special case, bring structure and sparsity to the web of dependence relations between many random variables \citep{LauritzenBook, wainwright+j:2008}. Extreme value theory for such models is a fairly recent subject. 
In \citep{asadi+d+s:2015}, a metric that takes the distance along a river into account underlies a spatial model for extremes of river networks. 
Recursive max-linear models on directed acyclic graphs are proposed in \citep{gissibl+k:2018} and put to work in \citep{einmahl+k+s:2018, gissibl+k+o:2018}.
In \citep{hitz+e:2016}, the density of a multivariate Pareto distribution is factorized through a version of the Hammersley--Clifford theorem. 
Such factorizations are also the theme in \citep{engelke+h:2018}, where they form the basis of new inference methods for extremes of graphical models, including the identification of the graphical structure itself.
Multivariate H\"usler--Reiss extreme-value copulas based on Gaussian Markov trees and higher-order truncated vines are introduced in \cite{lee+j:2018}, who propose composite likelihood methods based on bivariate margins to estimate the parameters.

Multivariate Pareto distributions arise as weak limits of normalized random vectors conditionally on the event that at least one component exceeds a high threshold. Although such conditioning events are covered by Theorem~\ref{thm:MPD:rho} below, the focus of this paper is rather on the case where the exceedance is known to have occurred at a specific variable. The message hinted at in the title is that both points of view are mathematically equivalent, but that, at least for Markov trees, the one-component limit is particularly elegant, as will be explained next.

%The same theme is treated in 
%In our setting, the global Markov property on the tree induces independence of the increments between the limit variables rather than between the variables themselves; see Section~\ref{subsec:tailtree} below. 
%In \citep{engelke+h:2018}, the Hammersley--Clifford theorem for absolutely continuous multivariate Pareto distributions is studied as well and is exploited for model construction and tail inference. Our Section~\ref{sec:ac} on absolutely continuous Markov trees lies at the intersection of their and our theory. It is an elegant special case but has the restrictive property that once a single variable is large, all variables must be large, precluding the existence of independent sources or factors causing extreme values. Multivariate Pareto distributions arise as weak limits when the conditioning event is that at least one variable exceeds a large threshold. Although we will include this case in Theorem~\ref{thm:MPD:rho} below, our focus is on the case where a specific variable is known to have exceeded a large threshold. Both points of view are equivalent, but, at least for Markov trees, the one-component limit is particularly elegant, as we will explain next.

%As the above explanations show, the particular case of Markov trees deserves to be worked out because of its resemblance to Markov chains.

% -------------------------------------
\subsection{Tail tree of a Markov tree}
%\label{subsec:tailtree}

For a Markov chain, it was discovered in \citep{smith:1992} that, conditionally on the event that the series is large at some time instant, the conditional distribution of the future of the system is that of a random walk, a process called tail chain in \citep{perfekt:1994}. For light-tailed marginal distributions, this random walk is additive, and for heavy-tailed margins it is geometric, i.e., multiplicative, which is the convention used in this paper.

A Markov tree can be viewed as a coupled collection of Markov chains with common stretches. Take for instance the four-variate Markov tree in Figure~\ref{fig:MT:4}. The nodes of the tree are $\{1,2,3,4\}$ and the three pairs of neighbours are $\{1,2\}$, $\{2,3\}$ and $\{2,4\}$. The vector $(X_1, X_2, X_3)$ is a Markov chain, and so is $(X_1, X_2, X_4)$. These two chains are coupled via the common pair $(X_1, X_2)$. Conditionally on $X_2$, the variables $X_1$, $X_3$ and $X_4$ are independent, since any path that connects two of the three nodes~1, 3 and 4 passes through node~2. This conditional independence property together with the distributions of the three pairs $(X_1, X_2)$, $(X_2, X_3)$ and $(X_2, X_4)$ determines the joint distribution of $(X_1, X_2, X_3, X_4)$.

For the moment, assume that the four variables have the same, regularly varying tail function. The set-up involving regular variation will be further motivated in Section~\ref{subsec:rv}. The effect (not necessarily causal) on $X_2$ of a large value at $X_1$ is via a multiplicative increment $M_{1,2}$ whose distribution is equal to the weak limit of $X_2/X_1$ conditionally on $X_1 = t$ as $t \to \infty$. The existence of this limit is an assumption on the Markov kernel induced by the distribution of the pair $(X_1, X_2)$. Similarly, a large value at $X_2$ affects $X_3$ and $X_4$ via the increments $M_{2,3}$ and $M_{2,4}$, respectively. The effect of $X_1$ on $X_3$ is then through the composite increment $M_{1,2} M_{2,3}$, whereas on $X_4$ it is through $M_{1,2} M_{2,4}$. The conditional independence property ensures that the increments $M_{1,2}$, $M_{2,3}$ and $M_{2,4}$ are mutually independent. The common edge $(1, 2)$ on the paths from node~$1$ to node~$3$ and from node~$1$ to node~$4$ induces dependence between the two tail chains $(M_{1,2}, M_{1,2} M_{2,3})$ and $(M_{1,2}, M_{1,2} M_{2,4})$ via the common increment $M_{1,2}$. In this paper, the random vector 
\begin{equation}
\label{eq:tt}
(\Theta_{1,2}, \Theta_{1,3}, \Theta_{1,4}) 
= (M_{1,2}, M_{1,2} M_{2,3}, M_{1,2} M_{2,4}) 
\end{equation}
is called the \emph{tail tree} induced by $X$ with root at node $u=1$.

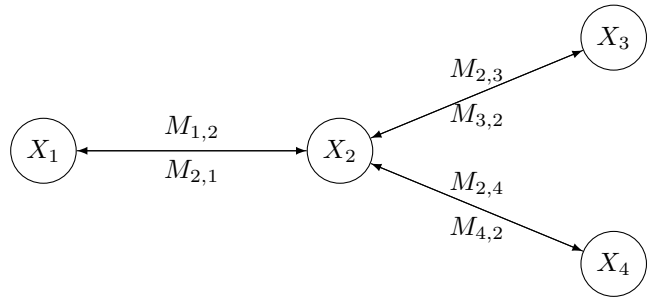
\begin{figure}
	\begin{center}
		\begin{tikzpicture}[scale=0.3]
		\node[draw, circle] (1) at (-13, 0) {$X_1$};
		\node[draw, circle] (2) at (0, 0) {$X_2$};
		\node[draw, circle] (3) at (12, 5) {$X_3$};
		\node[draw, circle] (4) at (12, -5) {$X_4$};
		
		\draw[->,>=latex] (1) -- (2) node[midway, above] {$M_{1,2}$};
		\draw[->,>=latex] (2) -- (3) node[midway, above] {$M_{2,3}$};
		\draw[->,>=latex] (2) -- (4) node[midway, above] {$M_{2,4}$};
		
		\draw[->,>=latex] (2) -- (1) node[midway, below] {$M_{2,1}$};
		\draw[->,>=latex] (3) -- (2) node[midway, below] {$M_{3,2}$};
		\draw[->,>=latex] (4) -- (2) node[midway, below] {$M_{4,2}$};
		\end{tikzpicture}
	\end{center}
	\caption{\label{fig:MT:4} A four-variate Markov tree $(X_1, \ldots, X_4)$ and the weak limits $M_{a,b}$ of $X_b/X_a$ given $X_a = t$ as $t \to \infty$ for edges $(a, b)$ in the tree. These limits constitute the multiplicative increments of the tail tree $\Theta_u$ starting at a given root, for instance at node $u = 1$ in \eqref{eq:tt} or at node $u = 3$ in \eqref{eq:tt:bis}.}
\end{figure}

The tail tree represents a network of stochastic dependence relations that are not necessarily causal. Suppose the Markov tree in Figure~\ref{fig:MT:4} represents water levels at four locations on a river network. If water flows from left to right, node~2 represents a point where the stream branches into two channels, as occurs for instance in a river delta. If water flows from right to left, however, node~2 represents the junction of two branches coming from nodes~3 and~4 into a larger stream flowing towards node~1. In the first case, the tail tree describes how a high water level at the upstream node~1 may cause high water levels at various locations in the delta further downstream. In the second, case however, it is nodes~3 and~4 that are situated upstream, and the tail tree models the sources of a high water volume at the downstream site~1. Still other set-ups are possible, such as for instance node~3 being upstream and nodes~1 and~4 being downstream: high water levels at nodes~1 and~4 are then related through a common cause at node~2, which can itself perhaps be traced back to node~3.

Whatever the causal relationships within $X$, it may make sense to change the conditioning variable. In Figure~\ref{fig:MT:4}, for instance, suppose it is known that a large value has occurred at node~3 rather than at node~1. Tracing the paths from node~3 to the three other nodes yields the tail tree with root at node $u = 3$:
\begin{equation}
\label{eq:tt:bis}
(\Theta_{3,2}, \Theta_{3,1}, \Theta_{3,4})
=
(M_{3,2}, M_{3,2}M_{2,1}, M_{3,2}M_{2,4}).
\end{equation}
The tail trees in \eqref{eq:tt} and~\eqref{eq:tt:bis} have a similar structure. The two edges on the path between the root nodes~1 and~3 have changed direction, however. The edge from node~2 to node~4 is common to both tail trees.

For each pair $\{a, b\}$ of neighbouring nodes, the choice of the root node~$u$ determines which of the two increments appears in the tail tree: $M_{a,b}$ from $X_a$ to $X_b$ or $M_{b,a}$ from $X_b$ to $X_a$. The distributions of $M_{a,b}$ and $M_{b,a}$ are connected by an expression that involves the marginal distributions of $X_a$ and $X_b$. For stationary and reversible Markov chains, this relation underlies a sufficiency property discovered in \citep{bortot+c:2000}. For tail chains of not necessarily reversible Markov chains, it was described in \citep{janssen+s:2014, segers:2007} and for tail processes of regularly varying stationary time series in \citep{basrak+s:2009} via the time change formula. This formula can be understood most easily  through the connection between the tail process and the tail measure \citep{dombry+h+s:2018, planinic+s:2018, samorodnitsky+o:2012}, and this is also the way in which the root change formula in Corollary~\ref{cor:modcons} below will be derived, but then without the assumption of stationarity and for general random vectors, not necessarily Markov trees.

% ----------------------------
\subsection{Regular variation}
\label{subsec:rv}

The language of regularly varying functions and measures provides a rich medium through which to express limit theorems. Recall that a positive, Lebesgue measurable function $f$ defined on a neighbourhood of infinity is regularly varying with index $\tau \in \reals$ if $\lim_{t \to \infty} f(\lambda t)/f(t) = \lambda^\tau$ for all $\lambda \in (0, \infty)$. If $X$ is a nonnegative random variable with unbounded support, cumulative distribution function $F(x) = \Pr(X \le x)$ and tail function $\Fbar = 1-F$, regular variation of $\Fbar$ with index $-\alpha < 0$ is equivalent to weak convergence of the conditional distribution of $X/t$ given that $X > t$ to a Pareto random variable $Y$ with index $\alpha$, i.e., $\Pr(Y > y) = y^{-\alpha}$ for all $y \in [1, \infty)$. We write $\law(X/t \mid X>t) \dto \Pa(\alpha)$ as $t \to \infty$, where $\law(Z\mid A)$ denotes the conditional distribution of the random object $Z$ given the event $A$, the arrow $\dto$ denotes convergence in distribution, and $\Pa(\alpha)$ denotes the said Pareto distribution.

For multivariate distributions, regular variation can be described via multivariate cumulative distribution functions as well, but an approach via convergence of Borel measures is more versatile. Let the state space be $\SS = [0, \infty)^d$. Generalizations to star-shaped metric spaces or abstract cones as in \citep{dombry+h+s:2018, HultLindskog2006, lindskog+r+r:2014, segers+z+m:2017} are left for further work. Let $I \subset \{1, \ldots, d\}$ denote the non-empty set of indices $i$ of variables of which the conditioning event $X_i > t$ is of possible interest. The marginal distributions of $X_i$ for $i \in I$ are assumed to be regularly varying and the ratios of their tail functions are assumed to converge to positive constants. This set-up is a bit more general than the one of identical margins and comes at little technical or notational cost.

The measures involved may have infinite mass but need to assign finite values to sets that remain bounded away from $\{ x \in \SS: \forall i \in I, x_i = 0 \}$ or $\{ x \in \SS : x_i = 0 \}$, depending on the conditioning event. The topology on the space of such measures will be the one proposed in \citep{lindskog+r+r:2014}, extending \citep{HultLindskog2006}, and resembles the one of vague convergence of measures, but avoiding the need to consider artificially compactified spaces. Regular variation is defined as convergence of $b(t) \, \Pr(X/t \in \point)$ to a limit measure called tail measure. Here, $b(t) > 0$ is a scale function tending to infinity and calibrated to the marginal distributions of $X_i$ for $i \in I$.

It is instructive to formulate statements in terms of weak convergence of distributions. For a high threshold $t$ tending to infinity and for a component $i \in I$, consider the asymptotic distribution of the rescaled random vector $X/t$ given that $X_i > t$. Decompose $X/t$ as $(X_i/t, X/X_i)$. Here, $X_i/t$ represents the overall level of $X$ with respect to $t$ whereas $X/X_i$ represents a self-normalized version of $X$. Convergence in distribution of $(X_i/t, X/X_i)$ given $X_i > t$ as $t \to \infty$ is a special case of what is called one-component regular variation in \citep{hitz+e:2016}, explored already in \citep{heffernan+r:2007, Resnick2014} for the bivariate case but allowing for affine normalizations. The random variable $X_i/t$ is asymptotically $\Pa(\alpha)$ distributed and independent of $X/X_i$, whose weak limit, denoted by $\Theta_i = (\Theta_{i,j})_{j=1}^d$, captures extremal dependence within $X$ given that $X_i$ is large. Letting the index $i$ run through $I$ produces multiple such one-component regular variation statements, which, together, are equivalent to what can be called multi-component regular variation. The limit distributions $\Theta_{i}$ that arise for various indices~$i$ must be mutually consistent, and the tail measure mentioned at the end of the previous paragraph embraces them all at once. 
%If $I = \{1, \ldots, d\}$, we are in the case of ordinary multivariate regular variation of $X$.

In Section~\ref{sec:one2multi}, the focus is on tying together multiple one-component regular variation limits. The theory is worked out for general random vectors, not necessarily Markov trees. A number of results in that section have already been formulated in the literature in one way or another, in slightly different settings. Some of the equivalence relations in Theorem~\ref{thm:one2multi}, for instance, resemble those in~\citep[Theorem~1.4]{hitz+e:2016} and~\citep[Proposition~3.1]{segers+z+m:2017}. The model consistency property between limit measures in Theorem~\ref{thm:one2multi}(ii) is formulated in \citep[Section~2]{das+r:2011} for the bivariate case. The root change formula in Corollary~\ref{cor:rcf} extends the time change formula for regularly varying stationary time series stemming from \citep{basrak+s:2009} and studied extensively in \citep{dombry+h+s:2018, janssen:2018}. Multivariate Pareto distributions as in Theorem~\ref{thm:MPD:rho} are foreshadowed in \citep[Section~6.3]{resnick:2006} and appear in \citep{ferreira+dh:2014, rootzen+s+w:2018} when $\rho(x) = \max(x_1, \ldots, x_d)$ and in \citep{dombry+r:2015} for more general functionals $\rho$.
%Furthermore, we did not even mention convergence of rescaled samples to Poisson processes, which figure prominently in \citep{resnick:2006} and many recent articles; see for instance \citep{basrak+p+s:2018} and the references therein. 
These are just a few connections, and the above list is by no means intended to be complete. 
%Still, we hope that Section~\ref{sec:one2multi} provides a useful overview. 

%We have assumed from the start that $X$ is $d$-dimensional, but we thereby also cover the finite-dimensional distributions of stochastic processes.
The set-up involving regular variation is intended to serve two purposes. First, to model tail dependence within a vector of random variables which have been transformed to the same, heavy-tailed distribution, such as the unit-Fréchet distribution, as is common in multivariate extreme value theory. Second, to model the joint distribution of a vector of regularly varying random variables, not necessarily identically distributed, but with equivalent tails, such as returns on financial portfolios composed of the same basket of underlying assets. The latter framework is more general than the former and comes at little additional notational cost.

% ------------------
\subsection{Outline}

For a Markov tree $X$, convergence as $t \to \infty$ of the conditional distribution of $X/X_u$ given that $X_u = t$ is proved in Section~\ref{sec:tailtree}. The main assumption is that, for edges $e = (a, b)$ directed away from the root $u$, the conditional distribution of $X_b/X_a$ given $X_a = t$ converges as $t \to \infty$. No regular variation is needed yet.

The tail trees pertaining to different roots $u$ can be linked up thanks to the theory of one- and multi-component regular variation developed in Section~\ref{sec:one2multi}. The results do not rely on the Markov property and cover quite general random vectors $X$ on $[0, \infty)^d$, as is illustrated briefly for max-linear models. An interesting special case of these are the recursive max-linear structural equation models introduced in \citep{gissibl+k:2018}, featuring a causal structure induced by a directed acyclic graph. Most of the proofs of this section are deferred to the Appendix.

When combined, the results in Section~\ref{sec:tailtree} and~\ref{sec:one2multi} serve to uncover the regular variation properties of Markov trees in Section~\ref{sec:rvmt}. 
%We assume regular variation of the marginal distributions at the nodes $u$ in a subset $U$ of the full node set $V$ and the earlier mentioned tail property of the Markov kernels associated to the edges $e = (a, b)$ that are directed outwards with respect to one or more such possible root nodes $u \in U$. 
The common special case that the joint distribution of the Markov tree is absolutely continuous with respect to Lebesgue measure is the subject of Section~\ref{sec:ac}. The theory then simplifies considerably and the limit distribution with respect to a single root $u$ is already sufficient to reconstruct the limit distributions with respect to all other possible roots $\bar{u}$. 

In Sections~\ref{sec:rvmt} and~\ref{sec:ac}, the distributions of the increments of the tail trees are calculated in case the pair distributions are max-stable, not necessarily absolutely continuous. For the H\"usler--Reiss distribution max-stable distribution, the tail tree is multivariate log-normal, constructed from partial sums of independent normal random variables along the edges of the tree.

\section{The spectral tail tree of a Markov tree}
\label{sec:tailtree}

%\js{In this section, we only need a directed tree with a fixed root.}
A (finite) graph is a pair $(V, E)$ where $V$ is a non-empty finite set of vertices or nodes and where $E \subset V \times V$ is a set of edges. Self-loops are excluded, i.e., $(u, u) \not\in E$ for all $u \in V$. To avoid trivialities, $V$ is assumed to have at least two elements. Two nodes are neighbours if they are joined by an edge. A graph is undirected if $(a, b) \in E$ implies $(b, a) \in E$. A path from a node $u$ to a node $v$ is a collection $\{e_1, \ldots, e_n\} \subset E$ of edges such that $e_k = (u_{k-1}, u_k)$ for all $k = 1, \ldots, n$, for $n+1$ \emph{distinct} nodes $u_0, u_1, \ldots, u_n \in V$ such that $u_0 = u$ and $u_n = v$. An undirected tree $\tree = (V, E)$ is an undirected graph such that for any pair of distinct nodes $u$ and $v$, there exists a unique path from $u$ to $v$, and this path is then denoted by $\path{u}{v}$.

Let $\tree = (V, E)$ be an undirected tree and let $X = (X_v)_{v \in V}$ be a random vector indexed by the nodes of the tree. The pair $(X, \tree)$ is a Markov tree if it satisfies the global Markov property \citep{LauritzenBook}: whenever $A, B, S$ are disjoint, non-empty subsets of $V$ such that $S$ separates $A$ and $B$ (i.e., any path between a node $a \in A$ and a node $b \in B$ passes through some node in $S$), the conditional independence relation
\begin{equation}
\label{eq:Markov}
X_A \indep X_B \mid X_S
\end{equation}
holds, where $X_W$ denotes the random vector $(X_v)_{v \in W}$ for $W \subset V$.

For an undirected tree $\tree = (V, E)$ and a node $u \in V$, let $\tree_u = (V, E_u)$ denote the directed, rooted tree that consists of directing the edges in $E$ outward starting from $u$. Formally, $E_u$ is the subset of $E$ that is obtained by choosing for every pair of edges $(a, b)$ and $(b, a)$ in $E$ the one such that the first node separates the second one from $u$. If $(a, b) \in E_u$, then $a$ is the (necessarily unique) parent of $b$ in $\tree_u$ whereas $b$ is a child of $a$ in $\tree_u$.

Let $(X, \tree)$ be a nonnegative Markov tree, where $\tree = (V, E)$ is an undirected tree.

\begin{condition}
	\label{ass:MT}
	There exists $u \in V$ with the following two properties.
	\begin{compactenum}[(i)]
		\item
		For every directed edge $e = (a, b) \in E_u$, there exists a version of the conditional distribution of $X_b$ given $X_a$ and a probability measure $\mu_e$ on $[0, \infty)$ such that
		\begin{equation}
		\label{eq:kernel:limit}
		\law( X_b / x_a \mid X_a = x_a ) \dto \mu_e, \qquad x_a \to \infty.
		\end{equation}
		\item
		For edges $e = (a, b) \in E_u$ such that $a \ne u$ and such that there exists an edge $\bar{e} \in \path{u}{a}$ for which $\mu_{\bar{e}}( \{0\} ) > 0$, we have
		\begin{equation}
		\label{eq:kernel:control}
		\forall \eta > 0, \qquad
		\lim_{\delta \downarrow 0}
		\limsup_{x \to \infty}
		\sup_{\eps \in [0, \delta]}
		\Pr( X_b / x > \eta \mid X_a = \eps x )
		= 0.
		\end{equation}
		%		\js{Given a root $u$, we can weaken the assumption: it only needs to hold for those $a$ such that $\Pr[ M_{i,j} = 0 ] = 0$ for all $(i,j) \in E_u$ on the path between $u$ and $a$. See Remark~\ref{rem:A ii omitted}.}
	\end{compactenum}
\end{condition}

Assumption~\ref{ass:MT}(ii) is similar to \citep[equation~(3.4)]{perfekt:1994} and prevents non-extreme values to cause extreme ones. A similar assumption is \citep[equation~(2.4)]{segers:2007}, where it is illustrated \citep[Example~7.5]{segers:2007} what can go wrong without it.

\begin{theorem}
	\label{thm:MT}
	Let $(X, \tree)$ be a nonnegative Markov tree on $\tree = (V, E)$. Assume Condition~\ref{ass:MT}. Let $(M_e : e \in E_u)$ be a vector of independent random variables such that the law of $M_e$ is $\mu_e$ for all $e \in E_u$. Then
	\begin{equation}
	\label{eq:spectral}
	\law ( X / X_u \mid X_u = t )
	\dto 
	\Theta_u = ( \Theta_{u,v} )_{v \in V},
	\qquad t \to \infty,
	\end{equation}
	where $\Theta_{u,u} = 1$ and
	\begin{equation}
	\label{eq:Thetauv}
	\forall v \in V \setminus \{u\}, \qquad 
	\Theta_{u,v} = \prod_{e \in \path{u}{v}} M_e.
	\end{equation}
\end{theorem}

The random vector $(\Theta_{u,v})_{v \in V}$ is called the \emph{tail tree} of the Markov tree $(X_v)_{v \in V}$, adapting terminology for Markov chains in \citep{perfekt:1994}. In Figure~\ref{fig:T}, the tail tree is illustrated for a tree with seven nodes. For subvectors $(\Theta_{u,w})_{w \in W}$ where all nodes in $W$ lie on the same path starting at $u$, the structure of the tail tree is that of a geometric random walk; take for instance $u = 1$ and $W = \{1, 4, 5, 7\}$ in Figure~\ref{fig:T}. The tail tree couples several geometric random walks together through the common edges in the underlying paths: in the same figure, consider for instance the vectors indexed by $\{1, 4, 5, 7\}$ and by $\{1, 4, 6\}$, respectively, which share the initial edge $(1, 4)$.

\begin{figure}
	\begin{center}
		\begin{tabular}{@{}p{0.7\textwidth}@{}p{0.25\textwidth}}
			\begin{minipage}[c]{0.7\textwidth}
				\begin{tikzpicture}[scale=0.8]
				\begin{scope}[ scale = 2 ]
				\node[draw, circle, fill=gray!30] (X1) at (-0.4,0) {\footnotesize $\Theta_{1,1}$};
				\node[draw, circle] (X2) at (-1.4,1) {\footnotesize $\Theta_{1,2}$};
				\node[draw, circle] (X3) at (-1.4,-1) {\footnotesize $\Theta_{1,3}$};
				\node[draw, circle] (X4) at (1,0) {\footnotesize $\Theta_{1,4}$};
				\node[draw, circle] (X5) at (2,1) {\footnotesize $\Theta_{1,5}$};
				\node[draw, circle] (X6) at (2,-1) {\footnotesize $\Theta_{1,6}$};
				\node[draw, circle] (X7) at (3.4,1) {\footnotesize $\Theta_{1,7}$};
				\end{scope}
				% \foreach \from/\to in {X1/X2, X1/X3, X1/X4, X4/X5, X4/X6, X5/X7}
				%   \draw [->,>=latex] (\from) -- (\to);
				\draw[->,>=latex] (X1) -- (X2) node[near end, right] {\footnotesize $M_{1,2}$};
				\draw[->,>=latex] (X1) -- (X3) node[near end, right] {\footnotesize $M_{1,3}$};
				\draw[->,>=latex] (X1) -- (X4) node[midway, above] {\footnotesize $M_{1,4}$};
				\draw[->,>=latex] (X4) -- (X5) node[midway, left] {\footnotesize $M_{4,5}$};
				\draw[->,>=latex] (X4) -- (X6) node[midway, right] {\footnotesize $M_{4,6}$};
				\draw[->,>=latex] (X5) -- (X7) node[midway, above] {\footnotesize $M_{5,7}$};
				\end{tikzpicture}
			\end{minipage}
			&
			\begin{minipage}[c]{0.25\textwidth}
				\[ 
				\begin{array}{l}
				\Theta_{1,1} = 1 \\
				\Theta_{1,2} = M_{1,2} \\
				\Theta_{1,3} = M_{1,3} \\
				\Theta_{1,4} = M_{1,4} \\
				\Theta_{1,5} = M_{1,4} M_{4,5} \\
				\Theta_{1,6} = M_{1,4} M_{4,6} \\
				\Theta_{1,7} = M_{1,4} M_{4,5} M_{5,7}
				\end{array}
				\]
			\end{minipage}
		\end{tabular}
	\end{center}	
	\caption{\label{fig:T} Illustration of the tail tree in \eqref{eq:Thetauv}. To each edge $e = (a, b)$ in the tree rooted at $u = 1$, a random variable $M_{a,b}$ is associated. These variables are independent. The limit variable $\Theta_{u,v}$ is the product of the variables $M_e$ along the edges $e$ on the path from $u$ to $v$. The joint distribution of $\Theta_{u} = (\Theta_{u,v})_{v \in V}$ is the tail tree with root $u$.}
\end{figure}
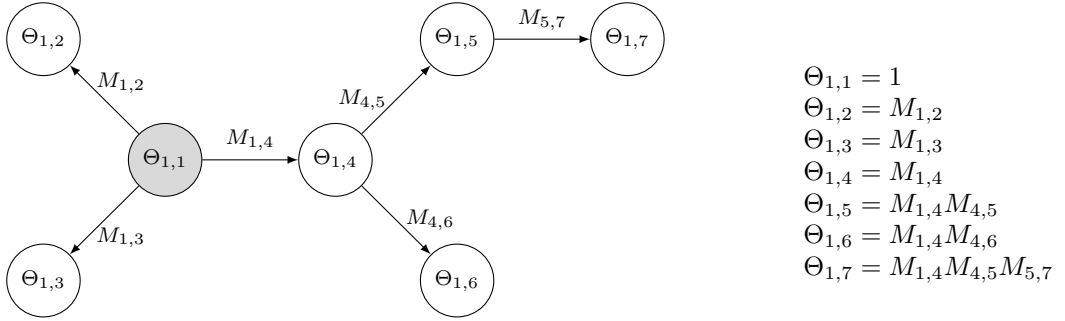

\begin{proof}[Proof of Theorem~\ref{thm:MT}]
	%	Let $f : [0, \infty)^V \to \reals$ be bounded and Lipschitz. We need to show that
	%	\[
	%	\lim_{x_0 \to \infty}
	%	\expec [ f( X / x ) \mid X_u = x ]
	%	=
	%	\expec [ f( \Theta_u ) ],
	%	\]
	%	with the random vector $\Theta_u$ defined as in \eqref{eq:Thetauv}. Without loss of generality, assume that $-1 \le f \le 1$.
	%	
	Put $d = \abs{V} - 1 \ge 1$. The proof is by induction on $d$.
	
	If $V$ has only two elements, i.e., $d = 1$, then Condition~\ref{ass:MT}(i) already confirms the convergence stated in \eqref{eq:spectral} and \eqref{eq:Thetauv}. Therefore, we can henceforth assume that $V$ has at least three elements, i.e., $d \ge 2$. Identify $V$ with $\{0, 1, \ldots, d\}$ in such a way that the root is $u = 0$ and such that if $(a, b) \in E_u$ then $a < b$. Since $X_0/X_0 = 1 = \Theta_{0,0}$, we do not need to consider the components $X_0$ and $\Theta_{0,0}$ in \eqref{eq:spectral}. 
	%	\js{Alternatively, let $w \in V$ be a terminal node in $\tree_u$, i.e., a node without children. Let $\bar{w} \in V$ be the parent of $w$ in $\tree_u$, i.e., $(\bar{w}, w) \in E_u$. Apply the induction hypothesis to the random vector $X_{-w} = (X_v)_{v \in V \setminus \{w\}}$, which is a Markov tree on the tree that is obtained from $\tree$ by omitting the node $w$ and the edge $(\bar{w}, w)$.}
	\medskip{}
	
	\noindent\emph{Step 1.} ---
	Let $k$ denote the parent of $d$ in the directed tree $\tree_0$, that is, $k$ is the unique node in $\{0, 1, \ldots, d-1\}$ such that $(k, d)$ is an edge in $E_0$. Our way of numbering nodes implies that $d$ cannot be the parent of every other node. Condition~\ref{ass:MT} is then satisfied also for the nonnegative Markov tree $X_{0:(d-1)} = (X_0, X_1, \ldots, X_{d-1})$ on the tree that is obtained from $\tree$ by removing node $d$ from $V$ and edges $(k, d)$ and $(d, k)$ from $E$. The induction hypothesis then means that, for every bounded, continuous function $g : [0, \infty)^d \to \reals$, we have
	\[
		\lim_{t \to \infty} \expec[ g(X_{1:(d-1)}/t) \mid X_0 = t] = \expec[ g(\Theta_{0,1:(d-1)}) ],
	\]
	the joint distribution of $\Theta_{0,1:(d-1)} = (\Theta_{0,1}, \ldots, \Theta_{0,d-1})$ being given by \eqref{eq:Thetauv}.

	%	Since $X_0 / x$ given $X_0 = 1$ is just $1$ almost surely and since $\Theta_{0, 0}$ is (by definition) equal to $1$ almost surely too, if suffices to show that, for every Lipschitz function $f : [0, \infty)^d \to [-1, 1]$, we have
	%	\[
	%		\lim_{x \to \infty} \expec[ f(x^{-1}X_{1:d}) \mid X_0 = x] = \expec[f(\Theta_{1:d})].
	%	\]
	%	
	
	Let $f : \reals^{d} \to [0, 1]$ be a Lipschitz function. We will show that
	\[
	\lim_{t \to \infty}
	\expec [ f( X_{1:d} / t ) \mid X_0 = t ]
	=
	\expec [ f( \Theta_{0,1:d}) ].
	\]
	Recall that $k \in \{0, 1, \ldots, d-1\}$ denotes the parent node of $d$. We need to distinguish between two cases: $k = 0$ is the root or $k \in \{1, \ldots, d-1\}$ is a non-root vertex. The case $k = 0$ is similar to but easier than the case $k \in \{1, \ldots, d-1\}$ and is left to the reader. We assume henceforth that $k \in \{1, \ldots, d-1\}$.
	\medskip{}
	
	\noindent\emph{Step 2.} ---
	Let $\delta > 0$ be such that $\Pr( \Theta_{0,k} = \delta ) = 0$. We have
	\begin{align}
	\label{eq:snoezie}
	\lefteqn{
		\left\lvert
		\expec[ f( X_{1:d}/t ) \mid X_0 = t ]
		-
		\expec[ f( \Theta_{0,1:d} ) ]
		\right\rvert
	} \\
	\label{eq:snoezie:up}
	&\le
	\left\lvert
	\expec[ \1( X_k/t \ge \delta) \, f( X_{1:d}/t ) \mid X_0 = t ]
	-
	\expec[ \1( \Theta_k \ge \delta) \, f( \Theta_{0,1:d} ) ]
	\right\rvert
	\\
	\label{eq:snoezie:down}
	&\quad\mbox{}+
	\left\lvert
	\expec[ \1( X_k/t < \delta) \, f( X_{1:d}/t ) \mid X_0 = t ]
	-
	\expec[ \1( \Theta_k < \delta ) \, f( \Theta_{0,1:d} ) ]
	\right\rvert.
	\end{align}
	We will show that the expression \eqref{eq:snoezie:up} converges to zero as $x_0 \to \infty$ (Step~3). Moreover, we will find a bound for the limit superior of \eqref{eq:snoezie:down} as $x \to \infty$. The bound will depend on $\delta$ but will converge to zero as $\delta \downarrow 0$ (Step~4). Together, these properties of \eqref{eq:snoezie:up} and \eqref{eq:snoezie:down} are sufficient to prove the theorem (Step~5).
	\medskip{}

	\noindent\emph{Step 3: The term \eqref{eq:snoezie:up}.} ---
	The vertex $k$ is the parent of $d$ in $\tree_0$, and therefore it separates $d$ from the other vertices. By the conditional independence property~\eqref{eq:Markov},
	\begin{multline*}
	\expec[ \1( \tfrac{X_k}{t} \ge \delta) \, f( \tfrac{X_{1:d}}{t} ) \mid X_0 = t ] \\
	=
	\int_{x_{1:(d-1)}}
	\1( \tfrac{x_k}{t} \ge \delta) \,
	\expec[ f( \tfrac{x_{1:(d-1)}}{t}, \tfrac{X_d}{t}) \mid X_k = x_k ]
	\Pr[ X_{1:(d-1)} \in \diff x_{1:(d-1)} \mid X_0 = t ].
	\end{multline*}
	To explain our notation: the integral is over $x_{1:(d-1)} = (x_1, \ldots, x_{d-1})$ and is with respect to the conditional distribution of $X_{1:(d-1)} = (X_1, \ldots, X_{d-1})$ given that $X_0 = t$. The integrand involves the conditional expectation of a function of $X_d$ given that $X_k = x_k$.
	
	We change variables and integrate with respect to the conditional distribution of $X_{1:(d-1)}/t$ given that $X_0 = t$: we get
	\begin{equation}
	\label{eq:bella:up}
	\expec[ \1( \tfrac{X_k}{t} \ge \delta) \, f( \tfrac{X_{1:d}}{t} ) \mid X_0 = t ] 
	=
	\int_{y_{1:(d-1)}}
	g_{t}(y_{1:(d-1)})
	\Pr[ \tfrac{X_{1:(d-1)}}{t} \in \diff y_{1:(d-1)} \mid X_0 = t ],
	\end{equation}
	where the integrand in \eqref{eq:bella:up} is given by
	\[
	g_{t}(y_{1:(d-1)})
	=
	\1(y_k \ge \delta) \expec[ f( y_{1:(d-1)}, y_k \tfrac{X_d}{t y_k}) \mid X_k = t y_k ].
	\]
	By Assumption~\ref{ass:MT}(i), we have $X_d / x_k \mid X_k = x_k \dto M_{k,d}$ as $x_k \to \infty$. Define
	\[
	g( y_{1:(d-1)} )
	=
	\1(y_k \ge \delta) \expec[ f( y_{1:(d-1)}, y_k M_{k,d} ) ].
	\]
	Recall that $f$ is bounded and (Lipschitz) continuous. By the extended continuous mapping theorem \citep[Theorem~18.11]{vandervaart:1998}, we have, for all vectors $y_{1:(d-1)}$ such that $y_k \ne \delta$ and for all functions $y_{1:(d-1)}(\point)$ such that $y_{1:(d-1)}(t) \to y_{1:(d-1)}$ as $t \to \infty$, the limit relation
	\[
	\lim_{t \to \infty} g_{t} \bigl( y_{1:(d-1)}(t) \bigr) = g( y_{1:(d-1)} ).
	\]
	Moreover, $\law(X_{1:(d-1)}/t \mid X_0 = t) \dto \law(\Theta_{1:(d-1)})$ as $t \to \infty$ by the induction hypothesis. By the same extended continuous mapping theorem, the integral \eqref{eq:bella:up} converges to
	\begin{multline*}
	\int_{y_{1:(d-1)}} g(y_{1:(d-1)}) \Pr[ \Theta_{1:(d-1)} \in \diff y_{1:(d-1)} ] \\
	= 
	\int_{y_{1:(d-1)}} 
	\1(y_k \ge \delta) 
	\expec[ f( y_{1:(d-1)}, y_k M_{k,d} ) ] 
	\Pr[ \Theta_{1:(d-1)} \in \diff y_{1:(d-1)} ].
	\end{multline*}
	Recall that $(M_e)_{e \in E_{u}}$ is a vector of independent random variables such that the law of $M_e$ is $\mu_e$ for $e \in E_u$. By construction, $M_{k,d}$ and $\Theta_{1:(d-1)}$ are then independent too: each component $\Theta_{0,j}$ of $\Theta_{1:(d-1)}$ is a product of random variables $M_{a,b}$ with $a, b \in \{0, \ldots, d-1\}$ and thus $e = (a, b) \neq (k, d)$. The above integral may therefore be simplified to
	\[
	\expec[ \1( \Theta_k \ge \delta ) \, f( \Theta_{1:(d-1)}, \Theta_k M_{k,d} ) ]
	= \expec[ \1( \Theta_k \ge \delta ) \, f( \Theta_{1:d} ) ]
	\]
	since $\Theta_d = \Theta_{k}M_{k,d}$. It follows that the limit of \eqref{eq:snoezie:up} as $t \to \infty$ is equal to zero.
	\medskip{}
	
	\noindent\emph{Step 4: The term \eqref{eq:snoezie:down}.} ---
	We consider two cases: $\Pr( \Theta_k = 0 ) = 0$ (Step~4.a) and $\Pr( \Theta_k = 0 ) > 0$ (Step~4.b).
	\medskip{}
	
	\noindent\emph{Step 4.a: The case $\Pr(\Theta_k = 0) = 0$.} ---
	Since $0 \le f \le 1$, the integral~\eqref{eq:snoezie:down} is bounded by
	\[
	\Pr( X_k/t < \delta \mid X_0 = t ) + \Pr( \Theta_k < \delta ).
	\]
	By the induction hypothesis, this sum converges to $2 \Pr( \Theta_k < \delta )$ as $x \to \infty$. The latter probability converges to zero as $\delta \downarrow 0$, as required.
	\medskip{}
	
	\noindent\emph{Step 4.b: The case $\Pr( \Theta_k = 0 ) > 0$.} ---
	We decompose \eqref{eq:snoezie:down} into three terms:
	\begin{align}
	\nonumber
	\lefteqn{
		\left\lvert
		\expec[ \1( \tfrac{X_k}{t} < \delta) \, f( \tfrac{X_{1:d}}{t} ) \mid X_0 = t ]
		-
		\expec[ \1( \Theta_k < \delta ) \, f( \Theta_{1:d} ) ]
		\right\rvert
	} \\
	&\le
	\label{eq:snoezie:down:1}
	\left\lvert
	\expec[ \1( \tfrac{X_k}{t} < \delta ) \, f( \tfrac{X_{1:d}}{t} ) \mid X_0 = t ]
	-
	\expec[ \1( \tfrac{X_k}{t} < \delta ) \, f( \tfrac{X_{1:(d-1)}}{t}, 0 ) \mid X_0 = t ]
	\right\rvert
	\\
	\label{eq:snoezie:down:2}
	&\quad\mbox{}
	+
	\left\lvert
	\expec[ \1( \tfrac{X_k}{t} < \delta ) \, f( \tfrac{X_{1:(d-1)}}{t}, 0 ) \mid X_0 = t ]
	-
	\expec[ \1( \Theta_k < \delta ) \, f( \Theta_{1:(d-1)}, 0 ) ]
	\right\rvert
	\\
	\label{eq:snoezie:down:3}
	&\quad\mbox{}
	+
	\left\lvert
	\expec[ \1( \Theta_k < \delta ) \, f( \Theta_{1:(d-1)}, 0 ) ]
	-
	\expec[ \1( \Theta_k < \delta ) \, f( \Theta_{1:d} ) ]
	\right\rvert.
	\end{align}
	Let $L > 0$ be such that $\abs{ f(y) - f(z) } \le L \sum_{j=1}^d \abs{y_j - z_j}$ for all $y, z \in \reals^d$. Furthermore, recall that $0 \le f \le 1$, so that also $\abs{ f(y) - f(z) } \le 1$ for all $y, z \in \reals^d$.
	\medskip{}
	
	\noindent\emph{Step 4.b.i: The term \eqref{eq:snoezie:down:1}.} ---
	The term \eqref{eq:snoezie:down:1} is bounded by
	\begin{equation}
	\label{eq:snoezie:down:1:0}
	\expec\left[ 
	\left.
	\1( \tfrac{X_k}{t} < \delta ) \, 
	\lvert f( \tfrac{X_{1:d}}{t} ) - f( \tfrac{X_{1:(d-1)}}{t}, 0 ) \rvert 
	\, \right| \,
	X_0 = t 
	\right]
	\le
	\expec[ \1( \tfrac{X_k}{t} < \delta ) \min(1, L \tfrac{X_d}{t}) \mid X_0 = t ].
	\end{equation}
	The node $k$ separates the nodes $0$ and $d$. By the global Markov property, the expectation on the right-hand side of \eqref{eq:snoezie:down:1:0} is therefore equal to
	\begin{equation}
	\label{eq:snoezie:down:1:aux}
	\int_{[0, \delta)} 
	\expec[ \min(1, L \tfrac{X_d}{t}) \mid X_k = \eps t ] \, 
	\Pr[ \tfrac{X_k}{t} \in \diff \eps \mid X_0 = t ].
	\end{equation}
	Let $\eta \in (0, 2/L)$. The conditional expectation in the integrand in \eqref{eq:snoezie:down:1:aux} satisfies
	\begin{align*}
	\expec[ \min(1, L X_d/t) \mid X_k = \eps t ]
	&=
	\expec[ 
	\min(1, L X_d/t) \, \1( X_d/t \le \eta) \mid X_k = \eps t 
	]\\
	&\quad\mbox{}
	+
	\expec[ 
	\min(1, L X_d/t) \, \1( X_d/t > \eta) \mid X_k = \eps t 
	]
	\\
	&\le
	L \eta + \Pr( X_d/t > \eta \mid X_k = \eps t ).
	\end{align*}
	Therefore, the integral in \eqref{eq:snoezie:down:1:aux} is bounded by
	\[
		L\eta + \sup_{\eps \in [0, \delta)} \Pr( X_d > \eta t \mid X_k = \eps t ).
	\]
	By Assumption~\ref{ass:MT}(ii), we can first take the limit superior as $t \to \infty$ and then the limit superior as $\delta \downarrow 0$ to find that
	\[
	\limsup_{\delta \downarrow 0} \limsup_{t \to \infty} \expec[ \1( X_k/t < \delta ) \min(1, LX_d/t) \mid X_0 = t ]
	\le L\eta.
	\]
	Since $\eta$ can be chosen arbitrarily close to zero, we find that the double limit superior above is equal to zero.
	\medskip{}
	
	\noindent\emph{Step 4.b.ii: The term \eqref{eq:snoezie:down:2}.} ---
	By the induction hypothesis, the term \eqref{eq:snoezie:down:2} converges to zero as $t \to \infty$.
	\medskip{}
	
	\noindent\emph{Step 4.b.iii: The term \eqref{eq:snoezie:down:3}.} ---
	Since $\Theta_d = \Theta_k M_{k,d}$, the term \eqref{eq:snoezie:down:3} is bounded by
	\[
	\expec[ \1( \Theta_k < \delta ) \, \lvert f( \Theta_{1:(d-1)}, 0 ) - f( \Theta_d ) \rvert ]
	\le
	\expec[ \1( \Theta_k < \delta ) \min( 1, L \Theta_k M_{k,d} ) ].
	\]
	By the dominated convergence theorem, the expectation on the right-hand side converges to zero as $\delta \downarrow 0$.
	\medskip{}
	
	\noindent\emph{Step 5.} ---
	The terms \eqref{eq:snoezie:up} and \eqref{eq:snoezie:down} were analyzed in Steps~3 and~4, respectively. In Step~3, it was shown that the term \eqref{eq:snoezie:up} converges to zero as $t \to \infty$, for any $\delta > 0$ such that $\Pr(\Theta_{0,k} = \delta) = 0$. In Step~4, it was shown that the limit superior as $t \to \infty$ of the term in \eqref{eq:snoezie:down} is bounded by a quantity depending on $\delta$ which converges to zero as $\delta \downarrow 0$. Since the expression in \eqref{eq:snoezie} does not depend on $\delta$, its limit as $t \to \infty$ must thus be zero.
	
	This completes the proof of the induction step and thus of the theorem.
\end{proof}

\begin{corollary}
	\label{cor:MT}
	In the setting of Theorem~\ref{thm:MT}, also $\law(X/X_u \mid X_u > t) \dto \Theta_u$ as $t \to \infty$.
\end{corollary}

\begin{proof}
	For a bounded and continuous function $f : [0, \infty)^V \to \reals$, we have
	\begin{multline*}
	\bigl\lvert \expec[ f(X/X_u) \mid X_u > t ] - \expec[ f(\Theta_u)] \bigr\rvert \\
	\le
	\frac{1}{\Pr(X_u > t)}
	\int_{(t,\infty)}
	\bigl\lvert 
	\expec[ f(X/X_u) \mid X_u = s ] 
	- 
	\expec[f(\Theta_u)] 
	\bigr\rvert \, 
	\Pr(X_u \in \diff s).
	\end{multline*}
	Given $\eps > 0$, Theorem~\ref{thm:MT} allows us to find $t(\eps)$ sufficiently high such that the absolute value inside the integral is bounded by $\eps$ for all $s \ge t(\eps)$. But then the left-hand side in the previous display is bounded by $\eps$ too, for all $t \ge t(\eps)$. Since $\eps > 0$ was arbitrary, the stated convergence in distribution follows.
\end{proof}

% =====================================================
\section{One- versus multi-component regular variation}
\label{sec:one2multi}

Let $X = (X_1, \ldots, X_d)$ be a random vector of nonnegative variables. Upon an obvious change in notation, Corollary~\ref{cor:MT} concerned weak convergence of $\law(X/X_i \mid X_i > t)$ as $t \to \infty$ for some $i \in \{1, \ldots, d\}$. This convergence plus regular variation of the marginal distribution of $X_i$ is a special case of what is called one-component regular variation in \citep{hitz+e:2016}. The weak limit, $\Theta_i = (\Theta_{i,j})_{j=1}^d$, depends on the choice of~$i$. There may be good reasons to consider these limits for several indices $i$. Let $I \subset \{1, \ldots, d\}$ be the set of all indices $i$ for which such a limit $\Theta_{i}$ exists. How are these random vectors $\Theta_{i}$ related?

In this section, several such one-component statements are combined into a single one which could be called multi-component regular variation. If $I = \{1,\ldots,d\}$, this is just ordinary multivariate regular variation. As discussed already in Section~\ref{subsec:rv}, the connections between the limits $\Theta_{i}$ generalize the time change formula for stationary regularly varying time series and can be deduced from their connections to a limiting tail measure.

Let $\SS = [0, \infty)^d$ for some positive integer $d$ and let $I \subset \{1,\ldots,d\}$ be non-empty. For $x \in \SS$, put $x_I = (x_i)_{i \in I}$. Define $\SzI = \{ x \in \SS : \max(x_I) > 0 \}$. Let $\MzI$ denote the collection of Borel measures $\nu$ on $\SzI$ with the property that $\nu(B)$ is finite for every Borel set $B$ of $\SzI$ that is contained in a set of the form $\{x \in \SS : \max(x_I) \ge \eps \}$ for some $\eps > 0$. Let $\mathcal{C}_{0,I}$ denote the collection of bounded, continuous functions $f : \SzI \to \reals$ for which there exists $\eps > 0$ such that $f(x) = 0$ as soon as $\max(x_I) \le \eps$. Let $\MzI$ be equipped with the smallest topology that makes the evaluation mappings $\nu \mapsto \nu(f) = \int f \, \diff \nu$ continuous, where $f$ ranges over $\mathcal{C}_{0,I}$. This is the notion of $\mathcal{M}_{\mathbb{O}}$ convergence in \citep{lindskog+r+r:2014}, with, in their notation, $\mathbb{C} = \{x \in [0, \infty)^d : \forall i \in I, x_i = 0 \}$ and $\mathbb{O} = \SS \setminus \mathbb{C} = \SzI$. The topology just defined is metrizable, turning $\MzI$ into complete, separable metric space, with convenient characterizations of relative compactness, a Portmanteau theorem, and a mapping theorem, all very much in the spirit of the notion of vague convergence of Borel measures on locally compact second countable Hausdorff spaces. Convergence of measures with respect to this topology is denoted by the arrow $\zto$. If $I$ is just a singleton, $\{i\}$ say, then notation is simplifed from $\SS_{0,\{i\}}$ to $\Szi$ and so on. 
%If $I = \{1, \ldots, d\}$, then we just write $\SS_{0}$ and so on. 
For $\alpha > 0$, let $\Pa(\alpha)$ denote the Pareto distribution on $[1, \infty)$ with shape parameter $\alpha$, that is, the distribution of a random variable $Z$ such that $\Pr(Z > z) = z^{-\alpha}$ for $z \ge 1$. Product measure is denoted by $\otimes$.

\begin{theorem}
	\label{thm:one2multi}
	Let $X = (X_1, \ldots, X_d)$ be a random vector in $\SS = [0, \infty)^d$ and let $I \subset \{1,\ldots,d\}$ be non-empty. Let $F_i(x) = \Pr(X_i \le x)$ and $\Fbar_i = 1-F_i$ for $i \in I$ and $x \in [0, \infty)$. Assume there exists a function $b$, regularly varying at infinity with index $\alpha > 0$, such that $\lim_{t \to \infty} b(t) \Fbar_i(t) = c_i \in (0, \infty)$ for $i \in I$. The following statements are equivalent:
	\begin{compactenum}[(a)]
		\item
		For every $i \in I$ we have $\law(X/X_i \mid X_i > t) \dto \law(\Theta_i)$ as $t \to \infty$ for some random vector $\Theta_i = (\Theta_{i,j})_{j=1}^d$ on $\SS$.
		\item
		For every $i \in I$ we have $\law(X_i/t, X/X_i \mid X_i > t) \dto \Pa(\alpha) \otimes \law(\Theta_i)$ as $t \to \infty$ for some random vector $\Theta_i$ on $\SS$.
		\item 
		For every $i \in I$ we have $\law(X/t \mid X_i > t) \dto \law(Y_i)$ as $t \to \infty$ for some random vector $Y_i = (Y_{i,j})_{j=1}^d$ on $\SS$.
		\item
%		For every $i \in I$ we have $b(t) \Pr(X/t \in \point) \zto \nu_i$ as $t \to \infty$ for some $\nu_i \in \Mzi$.
		For every $i \in I$ there exists $\nu_i \in \Mzi$ such that $b(t) \Pr(X/t \in \point) \zto \nu_i$ as $t \to \infty$ in $\Mzi$.
		\item
%		We have $b(t) \Pr(X/t \in \point) \zto \nu$ as $t \to \infty$ for some $\nu \in \MzI$.
		There exists $\nu \in \MzI$ such that $b(t) \Pr(X/t \in \point) \zto \nu$ as $t \to \infty$ in $\MzI$.
	\end{compactenum}
	In that case, the limiting objects are connected in the following ways: for all $i \in I$,
	\begin{compactenum}[(i)]
		\item 
		$Y_i$ is equal in distribution to $Y_{i,i} \Theta_i$, where $\law(Y_{i,i}) = \Pa(\alpha)$ and $Y_{i,i}$ and $\Theta_i$ are independent;
		\item
		$\nu_i$ is equal to the restriction of $\nu$ to $\Szi$;
		\item
		$\Pr(Y_i \in \point) = c_i^{-1} \nu(\point \cap \{ x : x_i > 1 \})$;
		\item
		for every Borel measurable $f : \SzI \to [0, \infty]$, we have
		\begin{equation}
		\label{eq:Th2nu}
		\int_{\SzI} f(x) \, \1 \{ x_i > 0 \} \, \diff \nu(x) \\
		%			=
		%			c_i \int_0^\infty \expec[ f(z \Theta_i) ] \, \alpha z^{-\alpha-1} \, \diff z
		=
		c_i \expec\left[ 
		\int_{0}^{\infty} 
		f(z \Theta_i) \, 
		\alpha z^{-\alpha-1} \, \diff z 
		\right].
		\end{equation}		
	\end{compactenum}
\end{theorem}

The proof of Theorem~\ref{thm:one2multi}, together with the proofs of the other theorems in this section, is given in Appendix~\ref{app:proofs}.

To highlight the connection with the theory of one-component regular variation in \citep{hitz+e:2016}, note that the random vector $(X, \boldsymbol{Y})$ taking values in $[1, \infty) \times \reals^{d-1}$ in \citep[Theorem~1.4]{hitz+e:2016} plays the same role as the random vector $(X_i, X/X_i)$ in Theorem~\ref{thm:one2multi} above. The equivalence between (ii) and (iv) in the cited theorem is then the same as the equivalence between (a) and (b) in Theorem~\ref{thm:one2multi}.

Further, note that in Theorem~\ref{thm:one2multi}(a), for $j \in \{1,\ldots,d\}$ such that $\Theta_{i,j}$ is not degenerate at $0$, we necessarily have $\liminf_{t \to \infty} \Pr(X_j > t) / \Pr(X_i > t) > 0$, i.e., the tail of $X_j$ is at least as heavy as the one of $X_i$. If also $j \in I$, we can reverse the roles of $i$ and $j$ to find that the condition that the tails of all variables $X_i$ with $i \in I$ are balanced is almost forced.

Apart from the characterizations (a)--(e) in Theorem~\ref{thm:one2multi}, other equivalent ones are possible, for instance, involving sequences rather than functions, with a scaling function inside the probability rather than outside, or with respect to radial and `angular' coordinates $(\rho(X)/t, X/\rho(X))$ for some appropriate functional $\rho$. See for instance \citep[Theorem~6.1]{resnick:2006} and \citep[Theorem~3.1]{lindskog+r+r:2014}. The tail measures $\nu$ and $\nu_i$ are homogeneous with index $-\alpha$ \citep[Theorem~3.1]{lindskog+r+r:2014} and, upon a coordinate transformation, can be written as product measures. Since the focus here is on the weak limits $\Theta_i$, these properties are not further elaborated upon. Statement~(e) in Theorem~\ref{thm:one2multi} implies that the vector $X_I = (X_i)_{i \in I}$ is multivariate regularly varying with limit measure $\nu_I(\point) = \nu( \{ x \in [0, \infty)^d : x_I \in \point \})$ on $[0, \infty)^I \setminus \{0\}$, which in turn implies, among other things, that it is in the domain of attraction of a multivariate max-stable distribution with Fr\'echet margins and exponent measure $\nu_I$; see for instance \citep{resnick:1987, resnick:2006}. 

A noteworthy special case of \eqref{eq:Th2nu} is when $f$ is the indicator function of the orthant $\{ x \in \SzI : \forall j \in J, \, x_j > y_j \}$, where $J \subset \{1, \ldots, d\}$ has a non-empty intersection with $I$ and where $y_j > 0$ for all $j \in J$. If $i \in I \cap J$, then $f(x) \1 \{ x_i > 0 \} = f(x)$, and thus
\begin{equation}
\label{eq:min}
\nu(\{x \in \SzI : \forall j \in J, \, x_j > y_j \})
=
c_i \, \expec[ \min \{ y_j^{-\alpha} \Theta_{i,j}^\alpha : j \in J \}].
\end{equation}
A remarkable consequence is that the right-hand side does not depend on the choice of $i \in I \cap J$. This invariance property is a special case of a more general mutual consistency property of the limit distributions $\law(\Theta_i)$ for $i \in I$ that is formulated in Corollary~\ref{cor:modcons} below.

%In \eqref{eq:min}, we can set $J = \{i\}$ for some $i \in I$. Since $\Theta_{i,i} = 1$ almost surely, the margins of $\nu$ are thus
%\begin{equation}
%\label{eq:nu:margins}
%	\forall i \in I, \, y \in (0, \infty), \qquad
%	\nu( \{ x \in \SzI : x_i > y \})
%	= c_i y^{-\alpha}.
%\end{equation}

\begin{corollary}[Model consistency]
	\label{cor:modcons}
	If the equivalent conditions of Theorem~\ref{thm:one2multi} are fulfilled, then, for Borel measurable $f : \SzI \to [0, \infty]$ and for $i, j \in I$, we have
	\begin{equation}
	\label{eq:modcons}
	%	c_i \int_0^\infty \expec[ f(z \Theta_i) \1 \{\Theta_{i,j} > 0 \}] \, \alpha z^{-\alpha-1} \, \diff z
	%	=
	%	c_j \int_0^\infty \expec[ f(z \Theta_j) \1 \{\Theta_{j,i} > 0 \}] \, \alpha z^{-\alpha-1} \, \diff z,
	c_j \expec\left[ \1 \{\Theta_{j,i} > 0 \}\int_0^\infty f(z \Theta_j) \, \alpha z^{-\alpha-1} \, \diff z \right]
	=
	c_i \expec\left[ \1 \{\Theta_{i,j} > 0 \} \int_0^\infty f(z \Theta_i) \, \alpha z^{-\alpha-1} \, \diff z \right].
	\end{equation}
\end{corollary}

\begin{proof}
	By \eqref{eq:Th2nu}, both sides in \eqref{eq:modcons} are equal to $\int_{\SzI} f(x) \, \1 \{ x_i > 0, x_j > 0 \} \, \diff \nu(x)$.
\end{proof}

\begin{corollary}[Root-change formula]
	\label{cor:rcf}
	If the equivalent conditions of Theorem~\ref{thm:one2multi} are fulfilled, then, for all Borel measurable $g : \SzI \to [0, \infty)$ and for all $i,j \in I$, we have
	\begin{equation}
	\label{eq:rcf1}
	c_{j} \, \expec[g(\Theta_j) \, \1 \{ \Theta_{j,i} > 0 \}]
	=
	c_{i} \, \expec[g(\Theta_{i} / \Theta_{i,j}) \, \Theta_{i,j}^{\alpha}].
	\end{equation}
	In particular, $\Pr[\Theta_{j,i} > 0] = 1$ if and only if $\expec[\Theta_{i,j}^\alpha] = c_j/c_i$, and then, for all $g$ as above,
	\begin{equation}
	\label{eq:rcf2}
	\expec[g(\Theta_j)]
	=
	\expec[g(\Theta_i / \Theta_{i,j}) \, \Theta_{i,j}^\alpha] / \expec[\Theta_{i,j}^\alpha].
	\end{equation}
\end{corollary}

\begin{proof}
	In Corollary~\ref{cor:modcons}, take $f(x) = g(x/x_j) \, \1 \{ x_j > 1 \}$. As $\Theta_{j,j} = 1$, the left-hand side of \eqref{eq:modcons} is $c_j \expec[\1\{\Theta_{j,i} > 0 \} g(\Theta_{j})]$. The right-hand side of \eqref{eq:modcons} becomes $c_i \expec[ \1 \{ \Theta_{i,j} > 0 \} g(\Theta_i/\Theta_{i,j}) \Theta_{i,j}^\alpha]$, in which the indicator function is redundant. 
	
	The special case $g \equiv 1$ in \eqref{eq:rcf1} yields $c_j \Pr(\Theta_{j,i} > 0) = c_i \, \expec[\Theta_{i,j}^\alpha]$. If $\expec[\Theta_{i,j}^\alpha] = c_j/c_i$, we have $\Pr(\Theta_{j,i}>0)=1$, and the indicator function on the left-hand side of \eqref{eq:rcf1} can be omitted.
\end{proof}

If the limit measure $\nu$ does not assign any mass to the coordinate hyperplane $\{x : x_i = 0\}$, the indicator in \eqref{eq:Th2nu} is redundant and $\nu$ can be expressed entirely in terms of $\law(\Theta_i)$. Moreover, whether this occurs or not can be read off from the $\alpha$-th moments of the components of $\Theta_i$.

\begin{corollary}
	\label{cor:justk}
	In Theorem~\ref{thm:one2multi}, we have, for $i \in I$,
	\begin{equation}
	\label{eq:nuxi0}
	\nu(\{x \in \SzI : x_i = 0 \})
	=
	\begin{cases}
	0 & \text{if $\expec[\Theta_{i,j}^\alpha] = c_j/c_i$ for all $j \in I$,} \\
	\infty & \text{otherwise}.
	\end{cases}		
	\end{equation}
	If $\nu(\{x : x_i = 0\}) = 0$ for some $i \in I$, then, for all Borel measurable $f : \SzI \to [0, \infty]$,
	\begin{equation}
	\label{eq:Thk2nu}
	\nu(f) = c_i \expec\left[ \int_0^\infty f(z \Theta_i) \, \alpha z^{-\alpha-1} \, \diff z \right].
	\end{equation}
	Moreover, all tail trees $\Theta_{j}$ for $j \in I$ are determined by $\Theta_{i}$ via \eqref{eq:rcf2}.
\end{corollary}

\begin{proof}
	Choose $i, j \in I$. In \eqref{eq:Th2nu}, let $f$ be the indicator function of the set $\{ x : x_i = 0 \}$ and let the index $i$ in \eqref{eq:Th2nu} be equal to the index $j$ chosen here.
	%For all $z > 0$, we have $\expec[f(z\Theta_i)] = \Pr(\Theta_{i,k} = 0)$. 
	It follows that $\nu(\{x \in \SzI : x_i = 0, x_j > 0 \})$ is zero if $\Pr(\Theta_{j,i} = 0) = 0$ and infinity otherwise.
	By \eqref{eq:rcf1} with $g \equiv 1$, we have $\Pr(\Theta_{j,i} = 0) = 0$ if and only if $\Pr(\Theta_{j,i} > 0) = 1$ if and only if $\expec[\Theta_{i,j}^\alpha] = c_j / c_i$. Equation~\eqref{eq:nuxi0} follows. If $\nu(\{x : x_i = 0\}) = 0$, we can omit the indicator $\1 \{ x_i > 0 \}$ on the left-hand side of \eqref{eq:Th2nu}, yielding~\eqref{eq:Thk2nu}.
\end{proof}

Let $f$ be the indicator function of the set $\{ x : \exists j \in J, \, x_j > y_j \}$, where $J \subset \{1,\ldots,d\}$ is non-empty and where $y_j > 0$ for all $j \in J$. If $\nu(\{x : x_i = 0\}) = 0$, then, by \eqref{eq:Thk2nu},
\begin{equation}
\label{eq:max}
\nu(\{ x \in \SzI : \exists j \in J, \, x_j > y_j \})
=
c_{i} \, \expec[ \max\{y_j^{-\alpha} \Theta_{i,j}^\alpha : j \in J\}].
\end{equation}
In contrast to equation~\eqref{eq:min}, equation~\eqref{eq:max} is true only when $\nu(\{x : x_i = 0\}) = 0$, a prerequisite for which Corollary~\ref{cor:justk} gives a necessary and sufficient condition.

By Corollary~\ref{cor:justk}, the case where $\expec[\Theta_{i,j}^\alpha] = c_{j}/c_{i}$ for all $j \in I$ leads to considerable simplifications. In fact, the special case $K = \{i\}$ in the next theorem implies that even the weak convergence of $\law(X/X_j \mid X_j>t)$ for $j \in I$ can then be deduced from the weak convergence of $\law(X/X_i \mid X_i>t)$ alone.

\begin{theorem}
	\label{thm:justK}
	In Theorem~\ref{thm:one2multi}, a sufficient condition for (a)--(e) to hold is that there exists a non-empty set $K \subset I$ with the following two properties:
	\begin{compactenum}[(i)]
		\item
		For every $i \in K$ we have $\law(X/X_i \mid X_i > t) \dto \law(\Theta_i)$ as $t \to \infty$ for some random vector $\Theta_i$ on $\SS$.
		\item
		For every $j \in I \setminus K$, there exists $i = i(j) \in K$ such that $\expec[ \Theta_{i,j}^\alpha ] = c_j / c_i$.
	\end{compactenum}
	In that case, also $\law(X/X_j \mid X_j > t) \dto \law(\Theta_j)$ as $t \to \infty$ for $j \in I \setminus K$, where the law of $\Theta_j$ is given in terms of the one of $\Theta_{i}$ with $i = i(j)$ via \eqref{eq:rcf2}.
%	\begin{equation}
%	\label{eq:Thetak2Thetai}
%	\Pr(\Theta_i \in \point) 
%	= %\frac%
%	{\expec[ \1 \{ \Theta_k / \Theta_{k,i} \in \point \} \, \Theta_{k,i}^\alpha]}%
%	/
%	{\expec[ \Theta_{k,i}^\alpha ]}
%	\end{equation}
%	with $k = k(i)$ as in (ii). In particular, $\Pr(\Theta_{i,k} = 0) = 0$.
\end{theorem}

The focus so far has been on weak limits of conditional distributions involving a high-threshold exceedance by a specific component. In the spirit of the multivariate peaks-over-thresholds methodology \citep{engelke+h:2018, kiriliouk+r+s+w:2018}, the following result covers, among other possibilities, the case where the conditioning event involves a high-threshold exceedance in at least one of a number of components.

\begin{theorem}
	\label{thm:MPD:rho}
	Suppose the conditions of Theorem~\ref{thm:one2multi} are fulfilled. Let $\rho : \SS \to [0, \infty)$ be continuous and homogeneous of order one, that is, $\rho(\lambda x) = \lambda \rho(x)$ for $\lambda \in [0, \infty)$ and $x \in \SS$. If $\SS_{\rho} :=\{ x \in \SzI : \rho(x) > 1 \}$ is contained in a set of the form $\{x : \max(x_I) > \eps \}$ for some $\eps > 0$ and if $\nu(\SS_{\rho}) > 0$, then $b(t) \Pr[\rho(X) > t] \to \nu(\SS_{\rho}) \in (0, \infty)$ as $t \to \infty$ and $\law(X/t \mid \rho(X) > t) \dto \nu(\point \cap \SS_{\rho})/\nu(\SS_{\rho})$ as $t \to \infty$. The conclusions of Theorem~\ref{thm:one2multi} thus apply to the random vector $(X, \rho(X))$ in the space $[0, \infty)^{d+1}$ and relative to the index set $I \cup \{d+1\}$.
\end{theorem}

Examples of $\rho$ in Theorem~\ref{thm:MPD:rho} are $\rho(x) = \max_{i \in I} a_i x_i$ and $\rho(x) = \sum_{i \in I} a_i x_i$, where $a \in [0, \infty)^I \setminus \{0\}$. The special case $I = \{1, \ldots, d\}$ and $\rho(x) = \max(x_1,\ldots,x_d)$ produces multivariate Pareto distributions as in \citep[Section~6.3]{resnick:2006} and other references mentioned in Section~\ref{subsec:rv}. Also covered by Theorem~\ref{thm:MPD:rho} is $\rho(x) = \min_{i \in J} a_i x_i$ for non-empty $J \subset I$ and $a_j \in (0, \infty)$ for all $j \in J$ provided $\expec[ \min \{ a_j^\alpha \Theta_{i,j}^\alpha : j \in J \}] > 0$ for some (and hence all) $i \in I$. If, however, $\nu(\SS_{\rho}) = 0$, then $\Pr[\rho(X) > t]$ decays more rapidly than $b(t)$, and more refined models are needed, opening up a whole new world of possibilities.

\begin{example}
	\label{ex:maxlinear}
%	To see the theory at work in a simple non-Markovian example, 
	Let $X = (X_1, \ldots, X_d)$ follow the max-linear model
	\begin{equation}
	\label{eq:maxlinear}
		\forall i \in \{1,\ldots,d\}, \qquad 
		X_i = \max_{r=1,\ldots,s} a_{i,r} Z_r,
	\end{equation}
	where $a_{i,r} \in [0, \infty)$ are scalars such that $\max_r a_{i,r} > 0$ for all $i$ and where $Z_1, \ldots, Z_s$ are independent and identically distributed nonnegative random variables whose common distribution function $F$ has a regularly varying tail function $\Fbar=1-F$ with index $-\alpha < 0$. The marginal tails satisfy $\Fbar_{i}(t) / \Fbar(t) \to \sum_r a_{i,r}^\alpha =: c_i$ as $t \to \infty$. If $X_i$ exceeds a large threshold $t \to \infty$, the probability that this was due to $Z_r$ is proportional to $a_{i,r}^\alpha$, and then the other factors $Z_{\bar{r}}$ for $\bar{r} \ne r$ are of smaller order than $Z_{r}$. It follows that (a) in Theorem~\ref{thm:one2multi} holds where the law of $\Theta_{i}$ is discrete with at most $s$ atoms and is given by
	\begin{equation}
	\label{eq:maxlinear:Theta}
		\law(\Theta_{i}) = \frac{1}{c_{i}} \sum_{r=1}^{s} a_{i,r}^\alpha \, \epsilon_{(a_{j,r}/a_{i,r})_{j=1}^d},
	\end{equation}
	with $\epsilon_x(\point)$ denoting a unit point mass at $x$. From \eqref{eq:maxlinear:Theta}, we find
	\begin{equation*}
%	\label{eq:maxlinear:toevallig}
		\expec[ \Theta_{i,j}^\alpha ] 
	%	= \sum_{\substack{j=1,\ldots,d\\a_{u,j} > 0}} \frac{a_{u,j}^\alpha}{c_{u}} \frac{a_{\bar{u},j}^\alpha}{a_{u,j}^\alpha}
		= \frac{1}{c_{i}} \sum_{r=1}^s a_{j,r}^\alpha \1 \{ a_{i,r} > 0 \}.
	\end{equation*}
	It follows that $\expec[\Theta_{i,j}^\alpha] = c_j/c_i$ as soon as $a_{i,r} > 0$ for every $r$ such that $a_{j,r} > 0$. Indeed, in that case, if $X_{j}$ is large, then some variable $Z_r$ with $r = 1, \ldots, s$ such that $a_{j,r}$ is positive was large, which in turn implies that $X_{i}$ is large as well, so that $\Pr(\Theta_{j,i} > 0) = 1$.
\end{example}

\begin{example}
	Recursive max-linear models on directed acyclic graphs were introduced in \citep{gissibl+k:2018}. Borrowing some of their notation, consider a directed acyclic graph $\mathcal{D} = (V, E)$ with nodes $V = \{1, \ldots, d\}$ and edges $E = \{(k,i) : i \in V, k \in \pa(i)\}$, where $\pa(i) \subset V$ denotes the possibly empty set of parents of $i$. Consider a random vector $X = (X_1, \ldots, X_d)$ given by the structural equation model
	\begin{equation}
	\label{eq:SEM}
		\forall i \in \{1, \ldots, d\}, \qquad
		X_i = \max \left[ \max \{\gamma_{ki} X_k : k \in \pa(i)\}, \gamma_{ii} Z_i \right],
	\end{equation}
	where the random variables $Z_1, \ldots, Z_d$ are as in Example~\ref{ex:maxlinear} with $s = d$ and where all coefficients $\gamma_{ki}$ and $\gamma_{ii}$ are (strictly) positive; the maximum over the empty set is zero by convention. Then by \citep[Theorem~2.2]{gissibl+k+o:2018}, the random vector $X$ admits the max-linear representation
	\[
		\forall i \in \{1, \ldots, d\}, \qquad
		X_i = \max_{j=1,\ldots,d} b_{ji} Z_j,
	\]
	with coefficients $b_{ji}$, for $i,j \in \{1,\ldots,d\}$, defined as follows: $b_{ii} = \gamma_{ii}$ and $b_{ji} = 0$ if $j \in V \setminus (\an(i) \cup \{i\})$, while
	\begin{equation}
	\label{eq:DAG:maxlinear}
		\forall i \in \{1,\ldots,d\}, \forall j \in \an(i), \qquad
		b_{ji} = \max_{p \in P_{ji}} \left\{\gamma_{jj} \prod_{e \in p} \gamma_e\right\},
	\end{equation}
	where $P_{ji}$ is the collection of paths $p = \{e_1, \ldots, e_n\}$ from $j$ to $i$ in $\mathcal{D}$; recall the definition of a path in the beginning of Section~\ref{sec:tailtree}. The representation in \eqref{eq:DAG:maxlinear} is of the form \eqref{eq:maxlinear} with $s = d$ and with $a_{i,r} = b_{ri}$ for $i,r \in \{1, \ldots, d\}$. It follows that $a_{i,r} = 0$ unless $r = i$ or $r \in \an(i)$.
	
	The condition that $a_{i,r} > 0$ whenever $a_{j,r} > 0$ is satisfied as soon as $j \in \an(i)$: indeed, in that case, we have $\an(j) \cup \{j\} \subset \an(i)$, so that $b_{rj} > 0$ implies $r \in \an(j) \cup \{j\}$ and thus $r \in \an(i)$, implying $b_{ri} > 0$. Through the structural equation model~\eqref{eq:SEM}, a large value appearing at a node~$j$ will also be felt at any of its descendants~$i$. We get $\Pr(\Theta_{j,i} > 0) = 1$ for $j \in \an(i)$, which, by Corollary~\ref{cor:rcf}, means that the root-change formula~\eqref{eq:rcf2} applies, by which the law of $\Theta_j$ can be recovered from the one of $\Theta_i$. Moreover, Theorem~\ref{thm:justK} applies with $K$ equal to the set of leaf nodes, i.e., the nodes without descendants.
	
	In the special case that the directed acyclic graph is also a directed, rooted tree, every node~$i$ has either exactly one parent or is equal to the root node, say $u$. In that case, the collection of paths $P_{ji}$ between $j \in \an(i)$ and $i \in \{1,\ldots,d\} \setminus \{u\}$ is a singleton, $p = \path{j}{i}$, and the formula for $b_{ji}$ in~\eqref{eq:DAG:maxlinear} simplifies to $b_{ji} = \gamma_{jj} \prod_{e \in \path{j}{i}} \gamma_e$. Furthermore, the tail tree $\Theta_u$ in~\eqref{eq:maxlinear:Theta} starting from the root node $u$ simplifies to the degenerate distribution at the point $\theta_u = (\theta_{u,1}, \ldots, \theta_{u,d})$ with coordinates $\theta_{u,j} = \prod_{e \in \path{u}{j}} \gamma_e \in (0, \infty)$ for $j \in \{1,\ldots,d\}$. This is of the form in~\eqref{eq:Thetauv} with degenerate increments $M_e = \gamma_e$ for all $e \in E$.
\end{example}

\section{Regularly varying Markov trees}
\label{sec:rvmt}

As in Section~\ref{sec:tailtree}, let $(X, \tree)$ be a nonnegative Markov tree on the undirected tree $\tree = (V, E)$. The general theory in Section~\ref{sec:one2multi} sheds light on the relation between two tail trees emanating at different roots.
%Building upon the previous results, we will compare the tail trees in \eqref{eq:Thetauv} originating from different roots. 
For two different nodes $u$ and $\bar{u}$ in $V$, the sets of directed edges $E_{u}$ and $E_{\bar{u}}$ are the same except for the edges connecting nodes on the path between $u$ and $\bar{u}$, which are directed in opposite ways in the two edge sets: For every $(a, b) \in \path{u}{\bar{u}} = E_u \setminus E_{\bar{u}}$, we have $(b, a) \in \path{\bar{u}}{u} = E_{\bar{u}} \setminus E_{u}$ and the other way around.

Condition~\ref{ass:MT} was formulated relative to a single root $u \in V$. The next condition covers all nodes $u \in U$ in a non-empty subset $U$ of $V$ as possible roots. For such $U$, let $E_U = \bigcup_{u \in U} E_u$ denote the set of directed edges that appear in at least one of the directed trees $E_u$.

\begin{condition}
	\label{ass:MT:all}
	There exists a non-empty $U \subset V$ with the following two properties:
	\begin{compactenum}[(i)]
		\item 
		For every $e = (a, b) \in E_U$, there exists a version of the conditional distribution of $X_b$ given $X_a$ and a probability measure $\mu_e$ on $[0, \infty)$ such that \eqref{eq:kernel:limit} holds.
		
		\item
		For every edge $e = (a, b) \in E_U$ for which there exists $u \in U$ such that $e \in E_u$ and an edge $\bar{e} \in \path{u}{a}$ such that $\mu_{\bar{e}}(\{0\}) > 0$, we have \eqref{eq:kernel:control}.
	\end{compactenum}
\end{condition}

If $u, v \in U$ in Condition~\ref{ass:MT:all} then every node $w \in V$ that is on the path between $u$ and $v$ can be added to $U$ and Condition~\ref{ass:MT:all} remains true. Indeed, for such $u,v,w$, we have $E_w \subset E_u \cup E_v$, which takes care of (i), and $\path{w}{a} \subset \path{u}{a} \cup \path{v}{a}$ for every node $a \in V$, which takes care of (ii). The author is grateful to an anonymous reviewer for having pointed this out.

%\begin{theorem}
%	\label{thm:rvmt}
%	Let $(X, \tree)$ be a nonnegative Markov tree on the undirected tree $\tree = (V, E)$. Assume that Conditions~\ref{cond:marginal:2} and~\ref{ass:MT:all} hold. Then for neighbouring nodes $a, b \in V$, the limit distributions $\mu_{(a, b)}$ and $\mu_{(b, a)}$ are connected through
%	\begin{equation}
%		\label{eq:muba}
%%		c_b \int_{(0, \infty)} f(x) \, \diff \mu_{(b, a)}(x)
%%		=
%%		c_a \int_{(0, \infty)} f(1/x) \, x^\alpha \, \diff \mu_{(a, b)}(x)
%		c_b \, \expec[ f(M_{(b,a)}) \, \1\{ M_{(b,a)} > 0 \}]
%		=
%		c_a \, \expec[ f(1/M_{(a,b)}) \, M_{(a,b)}^\alpha]
%	\end{equation}
%	for Borel measurable $f : (0, \infty) \to \reals$. Moreover, weak convergence \eqref{eq:spectral} to the limit random vector $\Theta_{u} = (\Theta_{u,v})_{v \in V}$ in \eqref{eq:Thetauv} holds for every $u \in V$, where $(M_e)_{e \in E}$ is a vector of independent random variables and $M_e$ has law $\mu_e$ for each $e \in E$. For two different roots $u, \bar{u} \in V$, the distributions of the tail trees $\Theta_{u} = (\Theta_{u,v})_{v \in V}$ and $\Theta_{\bar{u}} = (\Theta_{\bar{u},v})_{v \in V}$ are connected through the root-change formula \eqref{eq:rcf}. Finally, Condition~\ref{cond:gt:2} holds, so that the law of $X$ is multivariate regularly varying as in \eqref{eq:MRV} with limit measure $\nu$ satisfying \eqref{eq:nu}.
%\end{theorem}

\begin{corollary}
	\label{thm:rvmt}
	Let $(X, \tree)$ be a nonnegative Markov tree on the undirected tree $\tree = (V, E)$. Let $U \subset V$ be non-empty. Let $F_u(x) = \Pr(X_u \le x)$ and $\Fbar_u = 1-F_u$ for all $u \in U$. Assume that there exists a positive function $b$, regularly varying at infinity with index $\alpha > 0$, such that $b(t) \Fbar_u(t) \to c_u \in (0, \infty)$ as $t \to \infty$ for every $u \in U$. Assume that Condition~\ref{ass:MT:all} holds. Let $(M_e)_{e \in E_U}$ be a vector of independent random variables such that $M_e$ has law $\mu_e$ for each $e \in E_U$. Then all conclusions of Theorem~\ref{thm:one2multi} hold with $I = U$ and with $\Theta_{u}$ the tail tree in \eqref{eq:Thetauv} for $u \in U$.
\end{corollary}

\begin{proof}
	Condition~\ref{ass:MT:all} and Corollary~\ref{cor:MT} imply that assumption~(a) in Theorem~\ref{thm:one2multi} is satisfied for $I = U$ and with $\Theta_u$ the tail tree in \eqref{eq:Thetauv}, for every $u \in U$. All equivalence relations and other properties are then as stated in Theorem~\ref{thm:one2multi}.
\end{proof}

\begin{corollary}
	In Corollary~\ref{thm:rvmt}, if $a, b \in V$ are neighbours in $E$ and if they both belong to $U$, then the distributions of $M_{a,b}$ and $M_{b,a}$ mutually determine each other by
		\begin{equation}
		\label{eq:muba}
		%		c_b \int_{(0, \infty)} f(x) \, \diff \mu_{(b, a)}(x)
		%		=
		%		c_a \int_{(0, \infty)} f(1/x) \, x^\alpha \, \diff \mu_{(a, b)}(x)
			c_b \expec[ g(M_{b,a}) \, \1\{ M_{b,a} > 0 \}]
			=
			c_a \expec[ g(1/M_{a,b}) \, M_{a,b}^\alpha]
		\end{equation}
	for all Borel measurable $g : (0, \infty) \to [0, \infty]$.
\end{corollary}

\begin{proof}
	To find \eqref{eq:muba}, apply \eqref{eq:rcf1} to the case $d = 2$ and the random vector $(X_a, X_b)$. The two limit random vectors $\Theta_u$ in Theorem~\ref{thm:one2multi}(a) are $(1, M_{a,b})$ and $(M_{b,a}, 1)$ when conditioning on $X_a > t$ and on $X_b > t$, respectively. Equation~\eqref{eq:muba} implies $\Pr(M_{b,a} > z) = (c_a/c_b) \expec[ \1\{z M_{a,b} < 1\} \, M_{a,b}^\alpha]$ for all $z \in [0, \infty)$, so that the distribution of $M_{b,a}$ can be recovered from the one of $M_{a,b}$.
\end{proof}

For different roots $u, \bar{u} \in U$, the tail trees $\Theta_{u}$ and $\Theta_{\bar{u}}$ have the same multiplicative structure. The differences between their distributions lie in the starting nodes of the paths and in the distributions of the multiplicative increments for edges on the paths $\path{u}{\bar{u}}$ and $\path{\bar{u}}{u}$, since these edges change direction. For such edges of which the nodes belong to $U$ as well, the increment distributions are related by \eqref{eq:muba}. See Figure~\ref{fig:change} for an illustration.

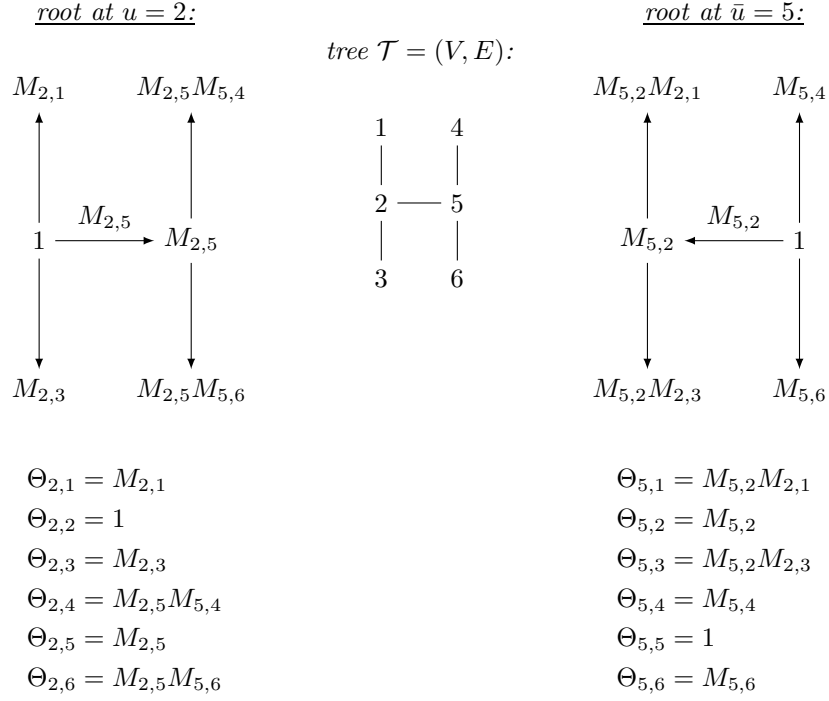
\begin{figure}
\begin{center}
	\begin{tikzpicture}
	\node (1) at (-0.5, 1.5) {$1$};
	\node (2) at (-0.5, 0.5) {$2$};
	\node (3) at (-0.5, -0.5) {$3$};
	\node (4) at (0.5, 1.5) {$4$};
	\node (5) at (0.5, 0.5) {$5$};
	\node (6) at (0.5, -0.5) {$6$};
	
	\draw (1) -- (2) -- (3);
	\draw (4) -- (5) -- (6);
	\draw (2) -- (5);
	
	\node at (0, 2.5) {\em tree $\mathcal{T}=(V, E)$:};
	
	\node (llu) at (-5,2) {$M_{2,1}$};
	\node (llc) at (-5,0) {$1$};
	\node (llb) at (-5,-2) {$M_{2,3}$};
	\node (lru) at (-3,2) {$M_{2,5} M_{5,4}$};
	\node (lrc) at (-3,0) {$M_{2,5}$};
	\node (lrb) at (-3,-2) {$M_{2,5} M_{5,6}$};
	
	\draw[->,>=latex] (llc) -- (llu);
	\draw[->,>=latex] (llc) -- (llb);
	\draw[->,>=latex] (lrc) -- (lru);
	\draw[->,>=latex] (lrc) -- (lrb);
	\draw[->,>=latex] (llc) -- (lrc) node[midway, above] {$M_{2,5}$};
	
	\node at (-4, 3) {\underline{\em root at $u = 2$:}};
	
	\node (rlu) at (3,2) {$M_{5,2} M_{2,1}$};
	\node (rlc) at (3,0) {$M_{5,2}$};
	\node (rlb) at (3,-2) {$M_{5,2} M_{2,3}$};
	\node (rru) at (5,2) {$M_{5,4}$};
	\node (rrc) at (5,0) {$1$};
	\node (rrb) at (5,-2) {$M_{5, 6}$};  
	
	\draw[->,>=latex] (rlc) -- (rlu);
	\draw[->,>=latex] (rlc) -- (rlb);
	\draw[->,>=latex] (rrc) -- (rru);
	\draw[->,>=latex] (rrc) -- (rrb);
	\draw[->,>=latex] (rrc) -- (rlc) node[midway, above] {$M_{5,2}$};
	
	\node at (4, 3) {\underline{\em root at $\bar{u} = 5$:}};
	\end{tikzpicture}
\end{center}
\begin{align*}
	\Theta_{2,1} &= M_{2,1} & \Theta_{5,1} &= M_{5,2} M_{2,1} \\
	\Theta_{2,2} &= 1 & \Theta_{5,2} &= M_{5,2} \\
	\Theta_{2,3} &= M_{2,3} & \Theta_{5,3} &= M_{5,2} M_{2,3} \\
	\Theta_{2,4} &= M_{2,5} M_{5,4} & \Theta_{5,4} &= M_{5,4} \\
	\Theta_{2,5} &= M_{2,5} & \Theta_{5,5} &= 1 \\
	\Theta_{2,6} &= M_{2,5} M_{5,6} \hspace{3cm} & \Theta_{5,6} &= M_{5,6}
\end{align*}
	\caption{\label{fig:change} The tail tree with root at node $u = 2$ (left) versus the tail tree with root at node $\bar{u} = 5$ (right) of the same Markov tree $X$ on the tree $\mathcal{T}$ depicted in the middle. The path $\path{u}{\bar{u}}$ is replaced by the path $\path{\bar{u}}{u}$ and the direction of any edge [here: just the edge $(2, 5)$] on the path between the two roots is reversed, while the other edges remain unaffected. For any edge $e = (a, b)$ that changes direction, the distribution of the increment $M_{a,b}$ changes to the one of the increment $M_{b,a}$. If both $a$ and $b$ belong to $U$, the increment distributions mutually determine each other through \eqref{eq:muba}.}
\end{figure}

For $u, \bar{u} \in U$, the equality $\expec[\Theta_{u,\bar{u}}^\alpha] = c_{\bar{u}}/c_{u}$ has interesting ramifications, see Corollaries~\ref{cor:rcf} and~\ref{cor:justk} and Theorem~\ref{thm:justK}. If all nodes on the path between $u$ and $\bar{u}$ belong to $U$ as well, then, since
\begin{equation}
\label{eq:Ma2Tha}
	\expec[ \Theta_{u,\bar{u}}^\alpha ] 
	= \prod_{e \in \path{u}{\bar{u}}} \expec[M_{e}^\alpha],
\end{equation} 
we have $\expec[\Theta_{u,\bar{u}}^\alpha] = c_{\bar{u}}/c_{u}$ as soon as $\expec[M_{a,b}^\alpha] = c_{b}/c_{a}$ for every $e = (a, b) \in \path{u}{\bar{u}}$.

Given the tree structure, the distribution of a Markov tree $X$ on $\tree = (V, E)$ is entirely determined by the bivariate distributions $(X_{a}, X_{b})$ for $e = (a, b) \in E$. Markov chains of which all pairs $(X_i, X_{i+1})$ are max-stable were proposed in \citep[Section~4.6]{coles+t:1991} and \citep{smith+t+c:1997}. When extended to trees, this construction method provides models meeting Condition~\ref{ass:MT:all}.

\begin{example}
	\label{ex:EVC:1}
	Let the distribution of the random pair $(X, Y)$ on $(0, \infty)^2$ be bivariate max-stable with cumulative distribution function
	\begin{equation*}
%	\label{eq:A2F}
		F(x, y) = \exp \{ - (x^{-1}+y^{-1}) \, A(x/(x+y)) \}, \qquad
		(x, y) \in (0, \infty)^2,
	\end{equation*}
	where $A : [0, 1] \to [1/2, 1]$ is a Pickands dependence function, that is, a convex function such that $\max(w, 1-w) \le A(w) \le 1$ for all $w \in [0, 1]$; see \citep{gudendorf_extreme-value_2010} and the references therein. Both marginal distributions are unit-Fr\'echet, $F(z,\infty)=F(\infty,z)=\exp(-1/z)$ for $z \in (0, \infty)$. In particular, the marginal tail functions are regularly varying at infinity with index $-\alpha=-1$. 
	
	Let $A'$ be the left-hand derivative of $A$, which exists everywhere on $(0, 1]$, takes values between $-1$ and $1$, and is non-decreasing and continuous from the left; define $A'(0)$ as the right-hand limit. Since $A$ is convex, it is absolutely continuous, and the set of points in $(0, 1)$ where it is not continuously differentiable is at most countable. For $x, y \in (0, \infty)$ such that $A$ is differentiable at $w = x/(x+y)$, we have
	\[
		\Pr(Y \le y \mid X = x) 
		= \frac{\partial F(x, y) / \partial x}{\partial F(x, \infty) / \partial x}
		= \exp \{ - x^{-1} ((1 - w)^{-1} A(w) - 1)\} \, \{ A(w) - w \, A'(w) \}.
	\]
	It follows that $\law(Y/x \mid X = x) \dto M$ as $x \to \infty$, where
	\begin{equation}
	\label{eq:A2M}
		\Pr(M \le z) = A(w) - w \, A'(w), \qquad z \in [0, \infty), \ w = 1 / (1+z).
	\end{equation}
	This is part~(i) of Condition~\ref{ass:MT}. Further, equation~\eqref{eq:kernel:control} in part~(ii) of Condition~\ref{ass:MT} follows from the monotone regression dependence property of bivariate max-stable distributions established in \citep{gg:2000}, by which the supremum over $\eps$ in \eqref{eq:kernel:control} is attained in $\eps = \delta$ and the limit superior as $x \to \infty$ is bounded by $\Pr(M \ge \eta/\delta)$, which tends to $0$ as $\delta \downarrow 0$ for every fixed $\eta > 0$.
	
	This construction using bivariate max-stable distributions is in some sense generic. Given a random variable $M$ on $[0, \infty)$ with expectation $\expec(M) \le 1$, one can define a Pickands dependence function $A$ by $A(w) = 1 - \expec[\min(1-w, wM)]$ for $w \in [0, 1]$, and then \eqref{eq:A2M} holds. The extension to general exponents $\alpha$ and tail constants $c_u$ is straightforward.
%	\begin{equation}
%	\label{eq:M2A}
%		A(w) = 1 - \expec[\min(1-w, wM)], \qquad w \in [0, 1],
%	\end{equation}
%	and then \eqref{eq:A2M} holds. 
%	A particular case in the spirit of Example~\ref{ex:maxlinear} is the max-linear model
%	\begin{align*}
%		X &= \max\{ \theta_1 Z_1, (1-\theta_1) Z_2\}, &
%		Y &= \max\{ (1-\theta_2) Z_1, \theta_2 Z_2 \},
%	\end{align*}
%	with parameter vector $(\theta_1, \theta_2) \in [0, 1]^2$, where $Z_1$ and $Z_2$ are independent unit-Fr\'echet random variables. The joint cumulative distribution function of $(X, Y)$ is given by \eqref{eq:A2F}, with piecewise linear Pickands dependence function $A$ given by \eqref{eq:M2A}, where the law of $M$ is discrete with one or two atoms.
\end{example}

% ======================================
\section{Absolutely continuous case}
\label{sec:ac}

If the joint distribution of the Markov tree $X = (X_v)_{v \in V}$ on $\tree = (V, E)$ is absolutely continuous with respect to the Lebesgue measure on $[0, \infty)^V$, the formulations of the conditions and results simplify considerably. Let $f$ denote the joint probability density function of $X$ and let $f_v$, for $v \in V$, denote the marginal density of $X_v$.

%\js{Under construction.}
%Marginal tails: let $f$ be a univariate probability density function on $(0, \infty)$ and let $F$ be the corresponding cumulative distribution function. If $f$ is regularly varying at infinity with index $-\alpha - 1$ for some $\alpha > 0$, then the tail function $\Fbar(x) = \int_x^\infty f(t) \, \diff t \sim \alpha^{-1} x \, f(x)$ is regularly varying at infinity with index $-\alpha$.

%\begin{condition}
%	\label{cond:marginal:pdf}
%	There exists a positive function $g$ on $[0, \infty)$ that is regularly varying at infinity with index $-\alpha-1$ for some $\alpha > 0$ and such that 
%	\[ 
%		\forall v \in V, \qquad \lim_{t \to \infty} \frac{f_{v}(t)}{g(t)} = c_{v} \in (0, \infty).
%	\]
%\end{condition}

%Let $X = (X_v)_{v \in V}$ be a random vector of positive variables indexed by the nodes of an undirected tree $\tree = (V, E)$. Assume that the law of $X$ is absolutely continuous with Lebesgue density $f : (0, \infty)^V \to [0, \infty)$. For each $v \in V$, let $f_{v}$ denote the univariate probability density function of $X_v$.

By the Hammersley--Clifford theorem \citep[Theorem~3.9]{LauritzenBook}, $X$ is a Markov tree as soon as the joint density factorizes as
\begin{equation*}
%	\label{eq:HC}
	\forall x \in \reals^d, \qquad
	f(x) = \prod_{j=1}^v f_v(x_v) 
	\prod_{\substack{\{a, b\} \subset V:\\\text{$a$ and $b$ are neighbours}}} \frac{f_{a,b}(x_{a}, x_{b})}{f_{a}(x_{a}) f_{b}(x_{b})}.
\end{equation*}
The second product is over all unordered pairs of neighbours and $f_{a,b}$ denotes the bivariate density function of $(X_a, X_b)$.

%For $e = (a, b) \in E$, let $f_{a,b}$ denote the joint density of $(X_a, X_b)$. 
For $t \in (0, \infty)$ such that $f_a(t) \in (0,\infty)$, the density of $\law(X_b/t \mid X_a = t)$ is $t f_{a,b}(t, ty) / f_{a}(t)$ for $y \in (0, \infty)$. The following condition replaces Condition~\ref{ass:MT:all}.

\begin{condition}
	\label{ass:MT:pdf}
	For every $e = (a, b) \in E$, there exists a probability density function $q_{a,b}$ on $(0, \infty)$ such that
	\[
		\forall y \in (0, \infty), \qquad \lim_{t \to \infty} 
		\frac{t f_{a,b}(t, ty)}{f_{a}(t)} = q_{a,b}(y).
	\]
%	Moreover, $q_{a,b}$ is a probability density function on $(0, \infty)$, i.e., $\int_0^\infty q_{(a,b)}(y) \, \diff y = 1$.
\end{condition}

\begin{theorem}
%	\label{thm:rvmt:pdf}
	Let the random vector $X$ on $[0, \infty)^V$ be a Markov tree on the undirected tree $\tree = (V, E)$ with joint density function $f$. Assume there exists a positive function $g$, regularly varying at infinity with index $-\alpha-1 < -1$, such that $f_v(t)/g(t) \to c_v \in (0, \infty)$ as $t \to \infty$ for every $v \in (0, \infty)$. If Condition~\ref{ass:MT:pdf} holds, then the conditions of Corollary~\ref{thm:rvmt} are satisfied with $U = V$, the same constants $c_u$, and auxiliary function $b(t) = \alpha / \{t \, g(t)\}$. For all pairs of neighbours $a, b \in V$, the density of $M_{a,b}$ is $q_{a,b}$ and for almost every $y \in (0, \infty)$, we have
	\begin{equation}
	\label{eq:quba}
	c_b \, y^\alpha \, q_{b,a}(y) = c_a \, y^{-2} \, q_{a,b}(y^{-1}).
	\end{equation}
	Moreover, $\expec[\Theta_{u,v}^\alpha] = c_{v}/c_{u}$ for all $u, v \in V$, so that $b(t) \Pr(X/t) \zto \nu$ as $t \to \infty$, where $\nu \in \Mz$ satisfies
	\[
		\nu(f) = c_u \expec \left[ 
			\int_{0}^{\infty} 
				f(z\Theta_{u}) \,
			\alpha z^{-\alpha-1} \, \diff z 
		\right]
	\]
	for every $u \in V$ and for every Borel measurable $f : [0, \infty)^V \setminus \{0\} \to [0, \infty]$, with $\Theta_{u}$ the tail tree in \eqref{eq:Thetauv}. Moreover, all tail trees are connected through \eqref{eq:rcf2}.
\end{theorem}

\begin{proof}
	The function $f_v$ is regularly varying at infinity with index $-\alpha-1$ too. By Karamata's theorem \citep[Proposition~1.5.10]{BGT}, we have $t f_v(t) / \Fbar_v(t) \to \alpha$ and thus $\Fbar_v(t) / \{t g(t)\} \to c_v/\alpha$ as $t \to \infty$.
	
	Condition~\ref{ass:MT:all} with $U = V$ follows from Condition~\ref{ass:MT:pdf} and Scheff\'e's theorem. Part~(ii) of Condition~\ref{ass:MT:all} is void, since $\mu_e(\{0\}) = 0$ for every $e \in E$.
	
	If $a, b \in V$ are neighbours, we can apply \eqref{eq:muba} to $g(y) = \1_{(0, z)}(y)$, where $z \in (0, \infty)$, to find
	\begin{equation*}
	c_b \, \int_0^z q_{b,a}(y) \, \diff y
	=
	c_a \, \int_{1/z}^\infty y^\alpha \, q_{a,b}(y) \, \diff y \\
	=
	c_a \, \int_{0}^{z} y^{-\alpha-2} \, q_{a,b}(y^{-1}) \, \diff y.
	\end{equation*}	
	Since this is true for every $z \in (0, \infty)$, we must have $c_b \, q_{b,a}(y) = c_{a} \, y^{-\alpha-2} \, q_{a,b}(y^{-1})$ for almost every $y \in (0, \infty)$, whence \eqref{eq:quba}.
	
	Since $\mu_{(b,a)}$ does not have an atom at $0$, the identity \eqref{eq:muba} with $g = \1_{(0, \infty)}$ implies that $\expec[M_{(a,b)}^\alpha] = c_{b}/c_{a}$. Apply \eqref{eq:Ma2Tha} and the observation on the line just below that equation to see that $\expec[\Theta_{u,v}^\alpha] = c_{v}/c_{u}$ for all $u, v \in V$. By Corollary~\ref{cor:rcf}, all tail trees are then connected via \eqref{eq:rcf2}.
	
	Finally, $\Mz$-convergence to $\nu$ with the stated expression follows from Theorem~\ref{thm:one2multi} and Corollary~\ref{cor:justk}.
\end{proof}

\begin{example}%[Extreme-value transitions]
	In Example~\ref{ex:EVC:1}, assume that $A$ is twice continuously differentiable on $(0, 1)$ and that $A'(0) = -1$ and $A'(1) = 1$. The distribution of $(X, Y)$ is then absolutely continuous and the conditional density of $Y/x$ given that $X = x$ converges as $x \to \infty$ to the function
	\[
		q(z) = w^3 \, A''(w), \qquad z \in (0, \infty), \ w = 1 / (1+z).
	\]
	The conditions on $A$ imply that $\int_0^\infty q(z) \, \diff z = \int_0^1 w \, A''(w) \, \diff w = 1$ and $\int_0^\infty z \, q(z) \, \diff z = \int_0^1 (1-w) \, A''(w) \, \diff w = 1$, so that $q$ is a probability density function with first moment equal to $1$. Moreover, replacing the function $A$ by the Pickands dependence function $w \mapsto A(1-w)$ amounts to changing $q$ by the function $z \mapsto z^{-3} q(z^{-1})$, in line with \eqref{eq:quba} with $c_a = c_b = 1$ and $\alpha = 1$.
	
%	\js{To be rewritten in the light of Example~\ref{ex:EVC:1}.}
%	Assume \eqref{eq:HC} where the bivariate densities are given by the ones of bivariate extreme-value distribution functions \js{Index $e$ niet nodig?}
%	\[
%		F_e(x, y) = \exp \{ - (x^{-1} + y^{-1}) \, A_e(w) \},
%		\qquad (x, y) \in (0, \infty)^2, \ w = x / (x+y).
%	\]
%	Here, $A_e : [0, 1] \to [1/2,1]$ is the Pickands dependence function of $F_e$, and it is a convex function such that $\max(w, 1-w) \le A_e(w) \le 1$ for all $w \in [0, 1]$.  Condition~\ref{cond:marginal:pdf} holds with marginal densities $f_{v}(x) = x^{-2} \, \exp(-x^{-1})$ for all $x \in (0, \infty)$ and all $v \in V$, so $\alpha = 1$ and $c_{v} = 1$ for all $v \in V$, provided we choose $g$ as the common density.
%	
%	The function $F_e$ is absolutely continuous as soon as $A_e$ is twice continuously differentiable, which we assume. If, moreover, the (one-sided) derivatives of $A_e$ at the endpoints are $A_e'(0) = -1$ and $A_e'(1) = 1$, then we have $\int_0^1 w \, A_e''(w) \, \diff w = 1 = \int_0^1 (1-w) \, A_e''(w) \, \diff w$, and, after some tedious calculations, Condition~\ref{ass:MT:pdf} can be verified too, with
%	\[
%		q_e(z) = w^3 \, A_e''(w), \qquad z \in (0, \infty), \ w = 1 / (1+z).
%	\]
	
	An interesting example in this respect is the bivariate H\"{u}sler--Reiss distribution \citep{husslerReiss1989} with Pickands dependence function
	\[
		A(w) = 
		(1-w) \, \Phi\left( \lambda + \tfrac{1}{2\lambda} \log \tfrac{1-w}{w} \right)
		+
		w \, \Phi\left( \lambda + \tfrac{1}{2\lambda} \log \tfrac{w}{1-w} \right)
	\]
	for $0 < w < 1$. Here $\lambda \in (0, \infty)$ is a parameter and $\Phi(z) = \int_{-\infty}^z (2\pi)^{-1/2} \exp(-z^2/2) \, \diff z$ is the standard normal cumulative distribution function. After tedious calculations, we find that $q$ is given by the density of the lognormal random variable $M = \exp\{2\lambda(Z-\lambda)\}$, where $Z$ is a standard normal random variable. Note indeed that $\expec[M] = 1$. Moreover, the density function satisfies $z \, q(z) = z^{-2} q(z^{-1})$ for all $z \in (0, \infty)$, which is \eqref{eq:quba} with $q_{a,b} = q_{b,a} = q$ and $c_a = c_b = 1$ and $\alpha = 1$. This also follows from the symmetry of the H\"usler--Reiss Pickands dependence function, i.e., $A(w) = A(1-w)$ for all $w \in [0, 1]$, so that the pair $(X, Y)$ is exchangeable.
	
	If all neighbouring pairs $(X_a, X_b)$ for $(a, b) \in E$ of the Markov tree follow such H\"{u}sler--Reiss max-stable distributions, the joint distribution of the tail tree is multivariate log-normal, since $\log \Theta_{u,v} = \sum_{e \in \path{u}{v}} \log M_e$ for all $u, v \in V$, where the random variables $\log M_e$ are independent and normally distributed with expectation $-2\lambda_e^2$ and variance $4\lambda_e^2$, with dependence parameter $\lambda_e \in (0, \infty)$ for all $e \in E$.
%	\js{Add reference.} Then \js{Verwijs naar \eqref{eq:A2M}.}
%	\[
%		q_e(z) = \text{\js{calculation to be completed}}
%	\]
%	is the density of a log-normal distribution with unit expectation, i.e., the density of $M = \exp\{2\lambda_e(Z-\lambda_e)\}$, where $Z$ is a standard normal random variable. The joint distribution of the tail tree is multivariate log-normal, since $\log \Theta_{u,v} = \sum_{e \in [u \to v]} \log M_e$ for all $u, v \in V$, and the random variables $\log M_e$ are independent and normally distributed with expectation $-2\lambda_e^2$ and variance $4\lambda_e^2$.
\end{example}

\appendix

\section{Proofs for Section~\ref{sec:one2multi}}
\label{app:proofs}

\begin{proof}[Proof of Theorem~\ref{thm:one2multi}]
	\emph{(a) and (b) are equivalent.} ---
	Clearly, (b) implies (a). Conversely, assume (a); let us show (b). Let $z \in [1, \infty)$ and let $\theta \in \SS$ be such that $\Pr(\Theta_j=\theta_j)=0$ for all $j \in \{1, \ldots, d\}$. We have
	\begin{align*}
	\Pr(X_i/t > z, X/X_i \le \theta \mid X_i > t)
	&=
	\frac{b(zt) \Fbar_i(zt)}{b(t) \Fbar_i(t)} \,
	\frac{b(t)}{b(zt)} \,
	\Pr(X/X_i \le \theta \mid X_i/t > z) \\
	&\to
	z^{-\alpha} \Pr(\Theta_i \le \theta),
	\qquad t \to \infty.
	\end{align*}
	It follows that $\Pr(X_i/t \le z, X/X_i \le \theta \mid X_i > t) \to \Pr(Z \le z) \Pr(\Theta_i \le \theta)$ as $t \to \infty$, where $Z$ is a $\Pa(\alpha)$ random variable.
	%	Conversely, choose $z \in [1, \infty)$ and let $f : \SS \to \reals$ be bounded and continuous. We have
	%	\begin{align*}
	%		\expec[\1(X_i/t > z) \, f(X/X_i) \mid X_i > t]
	%		&=
	%		\frac{\Fbar_i(zt)}{\Fbar_i(t)} \, \expec[f(X/X_i) \mid X_i > t] \\
	%		&\to
	%		z^{-\alpha} \, \expec[f(\Theta_i)],
	%		\qquad t \to \infty.
	%	\end{align*}
	\smallskip
	
	\emph{(b) implies (c) and (i).} ---
	Since $(X/t) = (X_i/t) (X/X_i)$, statement~(b) and the continuous mapping theorem \citep[Theorem~2.3]{vandervaart:1998} imply that $\law(X/t \mid X_i > t)$ converges weakly to $Y_i = Z \Theta_i$, where $Z$ is a $\Pa(\alpha)$ random variable independent of $\Theta_i$. Since $\Theta_{i,i} = 1$ almost surely, we have $Y_{i,i} = Z$.
	\smallskip
	
	\emph{(c) implies (a).} ---
	Since $X/X_i = (X/t)/(X_i/t)$, statement (c) and the continuous mapping theorem imply statement (a) with $\Theta_i = Y_i / Y_{i,i}$.
	\smallskip
	
	\emph{(b) implies (d).} ---
	Define a Borel measure $\nu_i$ on $\Szi$ by
	\[
	\nu_i(\point)
	=
	c_i \int_0^\infty \Pr(z \Theta_i \in \point) \, \alpha z^{-\alpha-1} \, \diff z.
	\]
	If $B$ is a Borel subset of $\Szi$ contained in $\{ x \in \Szi : x_i \ge \eps \}$ for some $\eps > 0$, then $\Pr(z \Theta_i \in B) = 0$ as soon as $z < \eps$, since $\Theta_{i,i} = 1$ almost surely. As a consequence, $\nu_i(B) \le c_i \eps^{-\alpha}$ for such $B$. It follows that $\nu_i \in \Mzi$. 
	
	By linearity of the integral and by monotone convergence, we find that
	\begin{equation}
	\label{eq:Thetai2nui}
	\nu_i(f) = c_i \int_0^\infty \expec[f(z\Theta_i)] \, \alpha z^{-\alpha-1} \, \diff z
	\end{equation}
	for every nonnegative Borel measurable function $f$ on $\Szi$. The same expression is then true for real-valued Borel measurable functions $f$ on $\Szi$ for which at least one of the two integrals with $f$ replaced by $\abs{f}$ is finite. This includes bounded, Borel measurable functions that vanish on a set of the form $\{x \in \Szi : x_i \le \eps \}$ for some $\eps > 0$.
	
	Let $f \in \Czi$ and let $\eps > 0$ be such that $f(x) = 0$ as soon as $x_i \le \eps$. By (b), we have
	\begin{align*}
	b(t) \expec[ f(X/t) ]
	&=
	b(t) \expec[ f((X_i/t)(X/X_i)) \, \1(X_i/t > \eps)] \\
	&=
	b(\eps t) \Fbar_i(\eps t) \frac{b(t)}{b(\eps t)}
	\expec[ f(\eps(X_i/(\eps t))(X/X_i)) \mid X_i > \eps t] \\
	&\to
	c_i \eps^{-\alpha} \expec[f(\eps Z \Theta_i)],
	\qquad t \to \infty,
	\end{align*}
	where $Z$ is a $\Pa(\alpha)$ random variable, independent of $\Theta_i$. The limit is equal to
	\begin{equation*}
	c_i \eps^{-\alpha} \int_{1}^{\infty} \expec[f(\eps z \Theta_i)] \, \alpha z^{-\alpha-1} \, \diff z
	=
	c_i \int_{\eps}^{\infty} \expec[f(z \Theta_i)] \, \alpha z^{-\alpha-1} \, \diff z \\
	%		=
	%		c_i \int_{0}^{\infty} \expec[f(z \Theta_i)] \, \alpha z^{-\alpha-1} \, \diff z \\
	=
	\nu_i(f),
	\end{equation*}
	since $f(z \Theta_i) = 0$ almost surely whenever $z \le \eps$, as $\Theta_{i,i} = 1$ almost surely.
	\smallskip{}
	
	\emph{(d) implies (c).} ---
	For $z \in (0, \infty)$, we have $b(t) \Pr(X_i/t > z) %= \{b(t)/b(zt)\} b(zt) \Fbar_i(zt) 
	\to c_i z^{-\alpha}$ as $t \to \infty$, and thus $\nu_i(\{x : x_i > z\}) = c_i z^{-\alpha}$ by (d). For open $G \subset \reals^d$, the Portmanteau theorem \citep[Theorem~2.1(iii)]{lindskog+r+r:2014} yields
	\begin{align*}
	\liminf_{t \to \infty} \Pr(X/t \in G \mid X_i > t)
	&=
	\liminf_{t \to \infty} 
	\frac{1}{b(t) \Fbar_i(t)} b(t) \Pr(X/t \in G \cap \{ x : x_i > 1\}) \\
	&\ge
	c_i^{-1} \nu_i(G \cap \{ x : x_i > 1\})
	\end{align*}
	By the Portmanteau lemma for weak convergence \citep[Lemma~2.2]{vandervaart:1998} we obtain (c) where the law of $Y_i$ is $\Pr(Y_i \in \point) = c_i^{-1} \nu_i(\point \cap \{ x : x_i > 1 \})$.
	\smallskip
	
	\emph{(d) implies (e).} ---
	For every $z > 0$, we have
	\begin{align*}
	b(t) \Pr[\max(X_I)/t > z]
	&\le
	b(t) \sum_{i \in I} \Pr(X_i/t > z) \\
	&=
	\frac{b(t)}{b(zt)} 
	\sum_{i \in I} b(zt) \Fbar_i(zt)
	\to
	z^{-\alpha} \sum_{i \in I} c_i,
	\qquad t \to \infty.
	\end{align*}
	Since the limit is finite for every $z > 0$ and since it converges to zero as $z \to \infty$, it follows by the relative compactness criterion in \citep[Theorem~2.5]{lindskog+r+r:2014} that for every sequence $(t_n)_n$ tending to infinity, there exists a subsequence along which $b(t_n) \Pr(X/t_n \in \point)$ converges in $\MzI$. To show (e), we then need to show that these subsequence limits must coincide. To do so, we show that for every $f \in \CzI$, the limit of $b(t) \expec[f(X/t)]$ exists as $t \to \infty$. This fixes the value of the integral of such $f$ with respect to all subsequence limits, which then must be the same.
	
	For $\eps > 0$, let $h_\eps : [0, \infty) \to [0, 1]$ be the piece-wise linear function
	\[
	h_\eps(t) =
	\min\{ \max(2t/\eps - 1, 0), 1 \} =
	\begin{cases}
	0 & \text{if $t \in [0, \eps/2]$,} \\
	2t/\eps - 1 & \text{if $t \in [\eps/2, \eps]$,} \\
	1 & \text{if $t \in [\eps, \infty)$.}
	\end{cases}
	\]
	Put $\hbar_\eps = 1 - h_\eps$. Write $I = \{i_{1}, \ldots, i_{k}\}$. Then
	\begin{align*}
	1 &= h_\eps(x_{i_{1}}) + \hbar_\eps(x_{i_{1}}) \\
	&= h_\eps(x_{i_{1}}) 
	+ \hbar_\eps(x_{i_{1}}) h_\eps(x_{i_{2}}) 
	+ \hbar_\eps(x_{i_{1}}) \hbar_\eps(x_{i_{2}}) \\
	&= \ldots \\
	&= \sum_{\ell=1}^k \left( \prod_{m=1}^{\ell-1} \hbar_\eps(x_{i_{m}}) \right) h_\eps(x_{i_{\ell}})
	+ \prod_{\ell=1}^k \hbar_\eps(x_{i_{\ell}}).
	\end{align*}
	
	For $f \in \CzI$ we can find $\eps > 0$ such that $f(x) = 0$ if $\max(x_{i_{1}},\ldots,x_{i_{k}}) \le \eps$. Then $f(x) \prod_{\ell=1}^k \hbar_\eps(x_{i_{\ell}}) = 0$ for all $x$, and thus $f = \sum_{i \in I} f_i$ where, for $\ell \in \{1, \ldots, k\}$, we have
	\[
	f_{i_{\ell}}(x) = f(x) \left( \prod_{m=1}^{\ell-1} \hbar_\eps(x_{i_{m}}) \right) h_\eps(x_{i_{\ell}}).
	\]
	Each function $f_i$ belongs to $\CzI$ too but has moreover the property that $f_{i}(x) = 0$ as soon as $x_{i} \le \eps/2$. The restriction of $f_i$ to $\Szi$ thus belongs to $\Czi$. By (d),
	\begin{equation*}
	b(t) \expec[ f(X/t) ]
	=
	\sum_{i \in I} b(t) \expec[f_i(X/t)]
	\to
	\sum_{i \in I} \nu_{i}(f_{i}), \qquad t \to \infty.
	\end{equation*}
	The existence of a limit has thus been shown, and convergence in $\MzI$ to some measure $\nu$ as stated in (e) follows.
	\smallskip
	
	\emph{(e) implies (d), (ii), (iii) and (iv).} ---
	A function $f$ in $\Czi$ can be extended to a function in $\CzI$ denoted by the same symbol by putting $f(x) = 0$ for $x \in \SzI \setminus \Szi$. Hence, (e) implies (d), with $\nu_i$ as described in (ii).
	
	Statement (iii) follows from (ii) and the description of the law of $Y_i$ in terms of $\nu_i$ in the proof above of the implication that (d) implies (c).
	
	Similarly, (iv) follows from (ii), equation~\eqref{eq:Thetai2nui}, and Fubini's theorem.
\end{proof}

\begin{proof}[Proof of Theorem~\ref{thm:justK}]
	It is sufficient to show statement (a) in Theorem~\ref{thm:one2multi}. By property~(i) in Theorem~\ref{thm:justK}, the weak convergence in Theorem~\ref{thm:one2multi}(a) already holds for all $i \in K$, and we need to show that it also holds for all $j \in I \setminus K$. Choose $j \in I \setminus K$ and let $i = i(j) \in K$ be as in property~(ii) of Theorem~\ref{thm:justK}
	
	%	By Theorem~\ref{thm:one2multi} applied to $K$, we know that
	%	\[
	%		\law(X_k/t, X_i/X_k \mid X_k > t)
	%		\dto
	%		\Pa(\alpha) \otimes \law(\Theta_{k,i}),
	%		\qquad t \to \infty.
	%	\]
	%	Let $Y$ be a $\Pa(\alpha)$ random variable, independent of $\Theta_{k,i}$. For $\delta > 0$, we find
	%	\begin{align*}
	%		\Pr(X_k > \delta t \mid X_i > t)
	%		&=
	%		\frac{\Fbar_k(\delta t)}{\Fbar_i(t)}
	%		\Pr(X_i > t \mid X_k > \delta t) \\
	%		&=
	%		\frac{\Fbar_k(\delta t) / \Fbar(\delta t)}{\Fbar_i(t) / \Fbar(t)}
	%		\frac{\Fbar(\delta t)}{\Fbar(t)}
	%		\Pr\left[
	%			\delta \frac{X_k}{\delta t} \frac{X_i}{X_k} > 1
	%			\, \Big\vert \,
	%			X_k > \delta t
	%		\right]
	%		&\to
	%		\frac{c_k}{c_i} \delta^{-\alpha} \Pr(\delta Y \Theta_{k,i} > 1),
	%		\qquad t \to \infty.
	%	\end{align*}
	%	Since $Y^{-\alpha}$ is independent of $\Theta_{k,i}$ and uniformly distributed on $(0, 1)$, the limit is equal to
	%	\[
	%		\frac{c_k}{c_i} \delta^{-\alpha}
	%		\expec[ \min(\delta^\alpha \Theta_{k,i}^\alpha, 1) ]
	%		=
	%		\frac{c_k}{c_i} \expec[ \min(\Theta_{k,i}^\alpha, \delta^{-\alpha}) ].
	%	\]
	%	By monotone convergence and by property~(ii), we obtain
	%	\[
	%		\lim_{\delta \downarrow 0} \lim_{t \to \infty}
	%		\Pr(X_k > \delta t \mid X_i > t) = 1.
	%	\]	
	We will show that $\law(X/X_j \mid X_j > t)$ converges weakly as $t \to \infty$ to $\Theta_j$ whose law is defined in \eqref{eq:rcf2}. Let $G \subset \SS$ be open and let $\delta > 0$.
	%	By the Portmanteau lemma for weak convergence, it is sufficient to show that
	%	\begin{equation} 
	%	\label{eq:liminfmu}
	%		\liminf_{t \to \infty} \Pr(X/X_i \in G \mid X_i > t) 
	%		\ge \Pr(\Theta_i \in G). 
	%	\end{equation}
	%	
	%	Let $\delta > 0$. 
	We have
	\begin{align*}
	%		\lefteqn{
	\Pr(X/X_j \in G \mid X_j > t)
	%		} \\
	&\ge
	\Pr(X/X_j \in G, X_i > \delta t \mid X_j > t) \\
	&=
	\frac{b(\delta t) \Fbar_i(\delta t)}{b(t) \Fbar_j(t)}
	\,
	\frac{b(t)}{b(\delta t)}
	\,
	\Pr\left[
	\frac{X/X_i}{X_j/X_i} \in G, \,
	\delta \frac{X_i}{\delta t} \frac{X_j}{X_i} > 1
	\, \Bigg\vert \,
	X_i > \delta t
	\right].
	\end{align*}
	By Theorem~\ref{thm:one2multi} applied to $K$, we have $\law(X_i/s, X/X_i \mid X_i > s) \dto \Pa(\alpha) \otimes \law(\Theta_{i})$ as $s \to \infty$. Let $Z$ be a $\Pa(\alpha)$ random variable, independent of $\Theta_{i}$. By the Portmanteau lemma for weak convergence, we have
	\begin{align*}
	\liminf_{t \to \infty} \Pr(X/X_j \in G \mid X_j > t)
	&\ge
	\frac{c_i}{c_j} \delta^{-\alpha}
	\Pr[ \Theta_i/\Theta_{i,j} \in G, \, \delta Z \Theta_{i,j} > 1 ] \\
	&= \expec[ \1 \{ \Theta_i / \Theta_{i,j} \in G \} \min( \Theta_{i,j}^\alpha, \delta^{-\alpha} ) ] / \expec[ \Theta_{i,j}^{\alpha} ].
	\end{align*}
	The equality on the second line follows from (ii) and the fact that $Z^{-\alpha}$ is uniformly distributed on $(0, 1)$ and independent of $\Theta_i$.
	%	\begin{equation*}
	%		\frac{c_k}{c_i} \delta^{-\alpha}
	%		\int_0^1 \Pr[ \Theta_k/\Theta_{k,i} \in G, \, \delta^\alpha \Theta_{k,i}^\alpha > u ] \, \diff u
	%		&=
	%		\frac{c_k}{c_i} \delta^{-\alpha}
	%		\expec[ \1 \{ \Theta_k / \Theta_{k,i} \in G \} \min( \delta^\alpha \Theta_{k,i}^\alpha, 1) ] \\
	%		&=
	%		\frac{c_k}{c_i} \expec[ \1 \{ \Theta_k / \Theta_{k,i} \in G \} \min( \Theta_{k,i}^\alpha, \delta^{-\alpha} ) ].
	%		\expec[ \1 \{ \Theta_k / \Theta_{k,i} \in G \} \min( \Theta_{k,i}^\alpha, \delta^{-\alpha} ) ] / \expec[ \Theta_{k,i}^{\alpha} ].
	%	\end{equation*}
	Since $\delta > 0$ was arbitrary, the monotone convergence theorem yields $\liminf_{t \to \infty} \Pr(X/X_j \in G \mid X_j > t) \ge \expec[\1\{\Theta_{i}/\Theta_{i,j} \in G \} \, \Theta_{i,j}^\alpha] / \expec[\Theta_{i,j}^\alpha]$. 	Apply the Portmanteau lemma for weak convergence once more to obtain the stated weak convergence.
\end{proof}

\begin{proof}[Proof of Theorem~\ref{thm:MPD:rho}]
	The properties of $\rho$ imply that $\SS_{\rho}$ is open and non-empty and that $0 < \nu(\SS_{\rho}) < \infty$.	The boundary of $\SS_{\rho}$ is $\{ x \in \SzI : \rho(x_I) = 1 \}$, which is $\nu$-null set, since its $\nu$-measure is bounded by the sum over $i \in I$ of $\nu( \{ x \in \SzI : \rho(x_I) = 1, x_i > 0 \})$, which is zero by \eqref{eq:Th2nu}. Since $\rho(X_I) > t$ if and only if $X_I/t \in \SS_{\rho}$, the Portmanteau theorem \citep[Theorem~2.1(iv)]{lindskog+r+r:2014} implies $b(t) \Pr[\rho(X_I) > t] \to \nu(\SS_{\rho})$ as $t \to \infty$.
	
	Let $G \subset \reals^d$ be open. By (iii) of the same Portmanteau theorem, 
	\[
	\liminf_{t \to \infty} \Pr(X/t \in G \mid \rho(X_I) > t)
	=
	\liminf_{t \to \infty} 
	\frac%
	{b(t) \Pr(X/t \in G \cap \SS_{\rho})}%
	{b(t) \Pr(X \in \SS_{\rho})}
	\ge
	\nu(G \cap \SS_{\rho}) / \nu(\SS_{\rho}).
	\]
	The Portmanteau theorem for weak convergence implies the stated weak convergence of $\law(X/t \mid \rho(X) > t)$ as $t \to \infty$. This proves statement (c) in Theorem~\ref{thm:one2multi} for the enlarged random vector $(X, \rho(X))$.
\end{proof}

\acks

The author is grateful to two anonymous reviewers whose suggestions have led to various improvements throughout the text. The author also wishes to thank Stefka Asenova and Gildas Mazo for inspiring discussions.

\small
\bibliographystyle{apt}
\bibliography{biblio}

\end{document}